\documentclass[11pt]{article}
\usepackage{amssymb,amsmath,amsfonts,amsthm,mathrsfs}
\usepackage{graphicx,graphics,psfrag,epsfig,calrsfs,cancel}
\usepackage[colorlinks=true, pdfstartview=FitV, linkcolor=blue, citecolor=red, urlcolor=blue]{hyperref}

\usepackage{caption,subfig,color}
\usepackage{tikz}
\usetikzlibrary{calc}

\usepackage{geometry}
\newtheorem{assumption}{Assumption}

\newtheorem{lemma}{Lemma}
\newtheorem{proposition}{Proposition}
\newtheorem{theorem}{Theorem}

\newtheorem{definition}{Definition}

\newcommand{\N}{{\mathbb N}}
\newcommand{\Z}{{\mathbb Z}}
\newcommand{\R}{{\mathbb R}}
\newcommand{\C}{{\mathbb C}}

\newcommand{\LL}{{\mathbb L}}
\newcommand{\M}{{\mathbb M}}
\newcommand{\D}{{\mathbb D}}
\newcommand{\Dbar}{\overline{\mathbb D}}
\newcommand{\U}{{\mathcal U}}
\newcommand{\Ubar}{\overline{\mathcal U}}
\newcommand{\cercle}{{\mathbb S}^1}

\newcommand{\dps}{\displaystyle}
\newcommand{\Ng}{| \! | \! |}
\newcommand{\Nd}{| \! | \! |}

\begin{document}

\title{The Leray-G{\aa}rding method \\
for finite difference schemes. II. \\
Smooth crossing modes}

\author{Jean-Fran\c{c}ois {\sc Coulombel}\thanks{Institut de Math\'ematiques de Toulouse - UMR 5219, Universit\'e de Toulouse ; 
CNRS, Universit\'e Paul Sabatier, 118 route de Narbonne, 31062 Toulouse Cedex 9 , France. Research of J.-F. C. was supported 
by ANR project Nabuco, ANR-17-CE40-0025. Email: {\tt jean-francois.coulombel@math.univ-toulouse.fr}}}
\date{\today}
\maketitle

\begin{abstract}
In \cite{jfcX} a multiplier technique, going back to {\sc Leray} and {\sc G{\aa}rding} for scalar hyperbolic partial differential equations, 
has been extended to the context of finite difference schemes for evolutionary problems. The key point of the analysis in \cite{jfcX} 
was to obtain a discrete energy-dissipation balance law when the initial difference operator is multiplied by a suitable quantity. The 
construction of the energy and dissipation functionals was achieved in \cite{jfcX} under the assumption that all modes were separated. 
We relax this assumption here and construct, for the same multiplier as in \cite{jfcX}, the energy and dissipation functionals when some 
modes cross. Semigroup estimates for fully discrete hyperbolic initial boundary value problems are deduced in this broader context by 
following the arguments of \cite{jfcX}.
\end{abstract}
\bigskip


\bigskip
\bigskip


Throughout this article, we keep the same notation as in \cite{jfcX}. We introduce the subsets of the complex plane:
\begin{align*}
&\U \, := \, \{\zeta \in \C,|\zeta|>1 \}\, ,\quad \Ubar \, := \, \{\zeta \in \C,|\zeta| \ge 1 \}\, ,\\
&\D \, := \, \{\zeta \in \C,|\zeta|<1 \}\, ,\quad \cercle \, := \, \{\zeta \in \C,|\zeta|=1 \} \, ,\quad \Dbar \, := \, \D \cup \cercle \, .
\end{align*}
We let ${\mathcal M}_n ({\mathbb K})$ denote the set of $n \times n$ matrices with entries in ${\mathbb K} = \R \text{ or } \C$. If 
$M \in {\mathcal M}_n (\C)$, $M^*$ denotes the conjugate transpose of $M$. We let $I$ denote the identity matrix or the identity 
operator when it acts on an infinite dimensional space. We use the same notation $x^* \, y$ for the Hermitian product of two vectors 
$x,y \in \C^n$ and for the Euclidean product of two vectors $x,y \in \R^n$. The norm of a vector $x \in \C^n$ is $|x| := (x^* \, x)^{1/2}$. 
The induced matrix norm on ${\mathcal M}_n (\C)$ is denoted $\| \cdot \|$.

The letter $C$ denotes a constant that may vary from line to line or within the same line. The dependence of the constants on the various 
parameters is made precise throughout the text.

In what follows, we let $d \ge 1$ denote a fixed integer, which will stand for the dimension of the space domain we are considering. We shall 
use the space $\ell^2$ of square integrable sequences. Sequences may be valued in $\C^k$ for some integer $k$. Some sequences will be 
indexed by $\Z^{d-1}$ while some will be indexed by $\Z^d$ or a subset of $\Z^d$. We thus introduce some specific notation for the norms. 
Let $\Delta x_k>0$ for $k=1,\dots,d$ be $d$ space steps as considered herafter. We shall make use of the $\ell^2(\Z^{d-1})$-norm that we 
define as follows: for all $v \in \ell^2 (\Z^{d-1})$,
\begin{equation*}
\| v \|_{\ell^2(\Z^{d-1})}^2 \, := \, \left( \, \prod_{k=2}^d \Delta x_k \, \right) \, \sum_{\nu=2}^d \, \sum_{j_\nu \in \Z} \, |v_{j_2,\dots,j_d}|^2 \, .
\end{equation*}
The corresponding scalar product is denoted $\langle \cdot,\cdot \rangle_{\ell^2(\Z^{d-1})}$. Then for all integers $m_1 \le m_2$ in $\Z$, we set
\begin{equation*}
\Ng u \Nd_{m_1,m_2}^2 \, := \, \Delta x_1 \, \sum_{j_1 = m_1}^{m_2} \, \|u_{j_1,\cdot} \|_{\ell^2(\Z^{d-1})}^2 \, ,
\end{equation*}
to denote the $\ell^2$-norm on the set $[m_1,m_2] \times \Z^{d-1}$ ($m_1$ may equal $-\infty$ and $m_2$ may equal $+\infty$). The corresponding 
scalar product is denoted $\langle \cdot,\cdot \rangle_{m_1,m_2}$. Other notation throughout the text is meant to be self-explanatory.

\section{Introduction}
\label{intro}

This article is a sequel of our previous work \cite{jfcX} where we have developed a multiplier technique for finite difference schemes. The theory 
in \cite{jfcX} encompasses the well-known example of the leap-frog scheme for the transport equation. Our main motivation was to derive stability 
estimates for finite difference schemes with a method that bypasses as much as possible Fourier analysis. This was a first step towards later 
considering multistep time integration techniques with finite volume space discretizations on unstructured meshes. We extend the results of \cite{jfcX} 
by dropping a \emph{simplicity} assumption that was made in this work, which now allows us to consider crossing eigenmodes. Namely, the situation 
we consider here is the one where the latter crossing occurs in a smooth way. We also deal completely with the case of multistep schemes with two 
time levels for which the eigenmode crossing need not be smooth. In order to avoid repeating many arguments from \cite{jfcX}, we shall refer to this 
work whenever possible. We warn the reader that the introduction below is mostly the same as in \cite{jfcX} since the considered problem is the same 
and we have found it easier for the reader to recall all the assumptions needed in the proof of our main result (which is Theorem \ref{mainthm} below). 
The main difference lies in the statement of Assumption \ref{assumption1} below.

We now set some more notation. With $d \in \N^*$ being the considered space dimension, we let $\Delta x_1,\dots,\Delta x_d,\Delta t>0$ denote the space 
and time steps where the ratios, also known as the so-called {\sc Courant-Friedrichs-Lewy} parameters, $\lambda_k :=\Delta t/\Delta x_k$, $k=1,\dots,d$, 
are \emph{fixed} positive constants. We keep $\Delta t \in (0,1]$ as the only free small parameter and let the space steps $\Delta x_1,\dots,\Delta x_d$ vary 
accordingly. The $\ell^2$-norms with respect to the space variables have been previously defined and thus depend on $\Delta t$ and the CFL parameters 
through the cell volume (either $\Delta x_2 \cdots \Delta x_d$ on $\Z^{d-1}$ or $\Delta x_1 \cdots \Delta x_d$ on $\Z^d$). We always identify a sequence 
$w$ indexed by either $\N$ (for time), $\Z^{d-1}$ or $\Z^d$ (for space), with the corresponding step function. In particular, we shall feel free to take Fourier 
or Laplace transforms of such sequences.

For all $j\in \Z^d$, we set $j=(j_1,j')$ with $j':=(j_2,\dots,j_d) \in \Z^{d-1}$. We let $p,q,r \in \N^d$ denote some fixed multi-integers, and define $p_1,q_1,r_1$, 
$p',q',r'$ according to the above notation. We also let $s \in \N$ denote some fixed integer. This article is devoted to recurrence relations of the form:
\begin{equation}
\label{numibvp}
\begin{cases}
{\dps \sum_{\sigma=0}^{s+1}} Q_\sigma \, u_j^{n+\sigma} \, = \, \Delta t \, F_j^{n+s+1} \, ,& j' \in \Z^{d-1} \, ,\quad j_1 \ge 1\, ,\quad n\ge 0 \, ,\\
u_j^{n+s+1} +{\dps \sum_{\sigma=0}^{s+1}} B_{j_1,\sigma} \, u_{1,j'}^{n+\sigma} \, = \, g_j^{n+s+1} \, ,& j' \in \Z^{d-1} \, ,\quad j_1=1-r_1,\dots,0\, ,\quad n\ge 0 \, ,\\
u_j^n \, = \, f_j^n \, ,& j' \in \Z^{d-1} \, ,\quad j_1\ge 1-r_1\, ,\quad n=0,\dots,s \, ,
\end{cases}
\end{equation}
where the operators $Q_\sigma$ and $B_{j_1,\sigma}$ are given by:
\begin{equation}
\label{defop}
Q_\sigma \, := \, \sum_{\ell_1=-r_1}^{p_1} \, \sum_{\ell'=-r'}^{p'} \, a_{\ell,\sigma} \, {\bf S}^\ell \, ,\quad 
B_{j_1,\sigma} \, := \, \sum_{\ell_1=0}^{q_1} \, \sum_{\ell'=-q'}^{q'} \, b_{\ell,j_1,\sigma} \, {\bf S}^\ell \, .
\end{equation}
In \eqref{defop}, the $a_{\ell,\sigma},b_{\ell,j_1,\sigma}$ are \emph{real numbers} and are independent of the small parameter $\Delta t$ (they may depend on the 
CFL parameters though), while ${\bf S}$ denotes the shift operator on the space grid: $({\bf S}^\ell v)_j :=v_{j+\ell}$ for $j,\ell \in \Z^d$. We have also used the 
short notation
\begin{equation*}
\sum_{\ell'=-r'}^{p'} \, := \, \sum_{\nu=2}^d \, \sum_{\ell_\nu=-r_\nu}^{p_\nu} \, ,\quad 
\sum_{\ell'=-q'}^{q'} \, := \, \sum_{\nu=2}^d \, \sum_{\ell_\nu=-q_\nu}^{q_\nu} \, .
\end{equation*}
Namely, the operators $Q_\sigma$ and $B_{j_1,\sigma}$ only act on the spatial variable $j \in \Z^d$, and the index $\sigma$ in \eqref{numibvp} keeps track of the 
dependence of \eqref{numibvp} on the $s+2$ time levels involved at each time iteration.

The numerical scheme \eqref{numibvp} is understood as follows: one starts with $\ell^2$ initial data $(f_j^0)$, ..., $(f_j^s)$ defined on $[1-r_1,+\infty) \times \Z^{d-1}$. 
The source terms $(F_j^n)$ and $(g_j^n)$ in \eqref{numibvp} are given. Assuming that the solution $u$ has been defined up to some time index $n+s$, $n \ge 0$, 
then the first and second equations in \eqref{numibvp} should uniquely determine $u_j^{n+s+1}$ for $j_1 \ge 1-r_1$, $j' \in \Z^{d-1}$. The mesh cells associated with 
$j_1 \ge 1$ correspond to the \emph{interior domain} while those associated with $j_1 = 1-r_1,\dots,0$ represent the \emph{discrete boundary}. Recurrence relations 
of the form \eqref{numibvp} arise when considering finite difference approximations of hyperbolic initial boundary value problems, see \cite{gko}, which is our main 
motivation (the Dirichlet and extrapolation boundary conditions considered in \cite{CLneumann} are typical examples).  We wish to deal here simultaneously with 
explicit and implicit schemes and therefore make the following solvability assumption.

\begin{assumption}[Solvability of \eqref{numibvp}]
\label{assumption0}
The operator $Q_{s+1}$ is an isomorphism on $\ell^2 (\Z^d)$. Moreover, for all $F \in \ell^2 (\N^* \times \Z^{d-1})$ and for all $g \in \ell^2 ([1-r_1,0] \times \Z^{d-1})$, 
there exists a unique solution $u \in \ell^2([1-r_1,+\infty) \times \Z^{d-1})$ to the system
\begin{equation*}
\begin{cases}
Q_{s+1} \, u_j \, = \, F_j \, ,& j' \in \Z^{d-1} \, ,\quad j_1 \ge 1\, ,\\
u_j +B_{j_1,s+1} \, u_{1,j'} \, = \, g_j \, ,& j' \in \Z^{d-1} \, ,\quad j_1=1-r_1,\dots,0\, .
\end{cases}
\end{equation*}
\end{assumption}

\noindent The first and second equations in \eqref{numibvp} therefore uniquely determine $u_j^{n+s+1}$ for $j_1 \ge 1-r_1$ and $j' \in \Z^{d-1}$; one then proceeds 
to the following time index $n+s+2$. Existence and uniqueness of a solution $(u_j^n)$ in $\ell^2([1-r_1,+\infty) \times \Z^{d-1})^\N$ to \eqref{numibvp} follows from 
Assumption \ref{assumption0} as long as the source terms lie in the appropriate functional spaces, so the last requirement for well-posedness is continuous dependence 
of the solution on the three possible source terms $(F_j^n)$, $(g_j^n)$, $(f_j^n)$. This is a \emph{stability} problem for which several definitions can be chosen according 
to the functional framework. The following one dates back to \cite{gks} in one space dimension and to \cite{michelson} in several space dimensions.

\begin{definition}[Strong stability]
\label{defstab1}
The finite difference approximation \eqref{numibvp} is said to be "strongly stable" if there exists a constant $C$ such that for all $\gamma>0$ and all $\Delta t \in \, (0,1]$, 
the solution $(u_j^n)$ to \eqref{numibvp} with zero initial data (that is, $(f_j^0) =\dots =(f_j^s) =0$ in \eqref{numibvp}) satisfies the estimate:
\begin{multline}
\label{stabilitenumibvp}
\dfrac{\gamma}{\gamma \, \Delta t+1} \, \sum_{n\ge s+1} \, \Delta t \, {\rm e}^{-2\, \gamma \, n\, \Delta t} \, \Ng u^n \Nd_{1-r_1,+\infty}^2 
+\sum_{n\ge s+1} \, \Delta t \, {\rm e}^{-2\, \gamma \, n\, \Delta t} \, \sum_{j=1-r_1}^{p_1} \| u_{j_1,\cdot}^n \|_{\ell^2 (\Z^{d-1})}^2 \\
\le \, C \, \left\{ \dfrac{\gamma \, \Delta t+1}{\gamma} \, \sum_{n\ge s+1} \, \Delta t \, {\rm e}^{-2\, \gamma \, n\, \Delta t} \, \Ng F^n \Nd_{1,+\infty}^2 
+\sum_{n\ge s+1} \, \Delta t \, {\rm e}^{-2\, \gamma \, n\, \Delta t} \, \sum_{j_1=1-r_1}^0 \| g_{j_1,\cdot}^n \|_{\ell^2 (\Z^{d-1})}^2 \right\} \, .
\end{multline}
\end{definition}

The main contributions in \cite{gks,michelson} are to show that strong stability can be characterized by an \emph{algebraic condition} which is usually referred to as the 
Uniform {\sc Kreiss-Lopatinskii} Condition. We shall assume here from the start that \eqref{numibvp} is strongly stable. We can thus control, for zero initial data, $\ell^2$ 
type norms of the solution to \eqref{numibvp}. Our goal, as in \cite{jfcX}, is to understand which kind of stability estimate holds for the solution to \eqref{numibvp} when 
one considers \emph{nonzero} initial data $(f_j^0),\dots,(f_j^s)$ in $\ell^2$. We are specifically interested in showing \emph{semigroup} estimates for \eqref{numibvp}, 
that is in controlling the $\ell^\infty_n (\ell^2_j)$ norm of the solution to \eqref{numibvp} (which is stronger than the $\ell^2_n (\ell^2_j)$ control encoded in 
\eqref{stabilitenumibvp}). Our main assumption is the following. It is a relaxed version of the corresponding assumption in \cite{jfcX} where the roots of the dispersion 
relation \eqref{dispersion} below were assumed to be always simple.

\begin{assumption}[Stability for the discrete Cauchy problem]
\label{assumption1}
For $\kappa \in (\C \setminus \{ 0 \})^d$, let us set :
$$
\widehat{Q_\sigma}(\kappa) \, := \, \sum_{\ell=-r}^p \kappa^\ell \, a_{\ell,\sigma} \, ,
$$
where the coefficients $a_{\ell,\sigma}$ are the same as in \eqref{defop} and we use the classical notation $\kappa^\ell := \kappa_1^{\ell_1} \cdots \kappa_d^{\ell_d}$ for 
$\kappa \in (\C \setminus \{ 0 \})^d$ and $\ell \in \Z^d$. Then there exists a finite number of points $\underline{\kappa}^{(1)}, \dots, \underline{\kappa}^{(K)}$ in 
$(\cercle)^d$ such that the following properties hold:
\begin{itemize}
 \item if $\kappa \in (\cercle)^d \setminus \{ \underline{\kappa}^{(1)}, \dots, \underline{\kappa}^{(K)} \}$, the roots to the dispersion relation\footnote{From Assumption 
          \ref{assumption0}, we know that $Q_{s+1}$ is an isomorphism on $\ell^2(\Z^d)$, which implies by Fourier analysis that $\widehat{Q_{s+1}} (\kappa)$ does not 
          vanish for $\kappa \in (\cercle)^d$. In particular, the dispersion relation \eqref{dispersion} is a polynomial equation of degree $s+1$ in $z$ for any $\kappa \in 
          (\cercle)^d$.}:
\begin{equation}
\label{dispersion}
\sum_{\sigma=0}^{s+1} \, \widehat{Q_\sigma} (\kappa) \, z^\sigma \, = \, 0 \, ,
\end{equation}
          are simple and located in $\Dbar$.
 \item if $\kappa$ equals one of the $\underline{\kappa}^{(k)}$'s, the dispersion relation \eqref{dispersion} has one multiple root $\underline{z}^{(k)} \in \D$ (its multiplicity 
          is denoted $m_k$) and all other roots are simple.
 \item for all $k=1,\dots,K$, there exists a neighborhood $\mathcal{V}_k$ of $\underline{\kappa}^{(k)}$ in $\C^d$ and there exist holomorphic functions $z_1, \dots, 
          z_{m_k}$ on $\mathcal{V}_k$ such that
$$
z_1 (\underline{\kappa}^{(k)}) \, = \, \cdots \, = \, z_{m_k} (\underline{\kappa}^{(k)}) \, = \, \underline{z}^{(k)} \, ,
$$
          and for all $\kappa \in \mathcal{V}_k$, $z_1(\kappa),\dots,z_{m_k}(\kappa)$ are the $m_k$ roots to \eqref{dispersion} that are close to $\underline{z}^{(k)}$.
\end{itemize}
\end{assumption}

\noindent  Assumption \ref{assumption1} means that the dispersion relation \eqref{dispersion} can have multiple roots (for stability reasons, multiple roots may only belong 
to $\D$ and not to $\cercle$). When multiple roots occur, we only ask that the splitting of the multiple eigenvalue around each such point be smooth (analytic). The fact that 
we only consider one multiple root at a time is only a matter of clarity and notation. There is no doubt that more elaborate crossings (e.g., with one root remaining double 
along a submanifold of $(\cercle)^d$) could be considered by further refining the techniques developed below. Eventually, we observe that multiple roots of the dispersion 
relation \eqref{dispersion} occur for instance when one uses the Adams-Bashforth or Adams-Moulton time integration technique of order $3$ or higher, see \cite[Chapter III]{hnw} 
(which is the reason why extending the result of \cite{jfcX} was necessary). We now make the following assumption, which already appeared in \cite{gks,michelson} and 
several other works on the same topic.

\begin{assumption}[Noncharacteristic discrete boundary]
\label{assumption2}
For $\ell_1=-r_1,\dots,p_1$, $z \in \C$ and $\eta \in \R^{d-1}$, let us define
\begin{equation}
\label{defA-d}
a_{\ell_1}(z,\eta) \, := \, \sum_{\sigma=0}^{s+1} \, z^\sigma \, \sum_{\ell'=-r'}^{p'} \, a_{(\ell_1,\ell'),\sigma} \, {\rm e}^{i \, \ell' \cdot \eta} \, .
\end{equation}
Then $a_{-r_1}$ and $a_{p_1}$ do not vanish on $\Ubar \times \R^{d-1}$, and they have nonzero degree with respect to $z$ for all $\eta \in \R^{d-1}$.
\end{assumption}

\noindent Our main result is comparable with \cite[Theorem 3.3]{wu}, \cite[Theorems 2.4 and 3.5]{jfcag} and \cite{jfcX}. It shows that strong stability (or "GKS stability") in 
the sense of Definition \ref{defstab1} is a \emph{sufficient} condition for incorporating $\ell^2$ initial conditions in \eqref{numibvp} and proving \emph{optimal} semigroup 
estimates. Our result reads just as in \cite{jfcX} but it now holds in the broader context of Assumption \ref{assumption1}.

\begin{theorem}
\label{mainthm}
Let Assumptions \ref{assumption0}, \ref{assumption1} and \ref{assumption2} be satisfied, and assume that the scheme \eqref{numibvp} is strongly stable in the sense of 
Definition \ref{defstab1}. Then there exists a constant $C$ such that for all $\gamma>0$ and all $\Delta t \in \, (0,1]$, the solution to \eqref{numibvp} satisfies the estimate:
\begin{multline}
\label{estim1d}
\sup_{n \ge 0} \, {\rm e}^{-2\, \gamma \, n\, \Delta t} \, \Ng u^n \Nd_{1-r_1,+\infty}^2 
+\dfrac{\gamma}{\gamma \, \Delta t+1} \, \sum_{n\ge 0} \, \Delta t \, {\rm e}^{-2\, \gamma \, n\, \Delta t} \, \Ng u^n \Nd_{1-r_1,+\infty}^2 \\
+\sum_{n\ge 0} \, \Delta t \, {\rm e}^{-2\, \gamma \, n\, \Delta t} \, \sum_{j_1=1-r_1}^{p_1} \, \| u_{j_1,\cdot}^n \|_{\ell^2(\Z^{d-1})}^2 
\le C \, \left\{ \sum_{\sigma=0}^s \, \Ng f^\sigma \Nd_{1-r_1,+\infty}^2 
+\dfrac{\gamma \, \Delta t+1}{\gamma} \, \sum_{n\ge s+1} \, \Delta t \, {\rm e}^{-2\, \gamma \, n\, \Delta t} \, \Ng F^n \Nd_{1,+\infty}^2 \right. \\
\left. +\sum_{n\ge s+1} \, \Delta t \, {\rm e}^{-2\, \gamma \, n\, \Delta t} \, \sum_{j_1=1-r_1}^0 \, \| g_{j_1,\cdot}^n \|_{\ell^2(\Z^{d-1})}^2 \right\} \, .
\end{multline}
In particular, the scheme \eqref{numibvp} is "semigroup stable" in the sense that there exists a constant 
$C$ such that for all $\Delta t \in \, (0,1]$, the solution $(u_j^n)$ to \eqref{numibvp} with $(F_j^n) =(g_j^n) 
=0$ satisfies the estimate
\begin{equation}
\label{estimsemigroup}
\sup_{n\ge 0} \, \Ng u^n \Nd_{1-r_1,+\infty}^2 \, \le \, C \, \sum_{\sigma=0}^s \, \Ng f^\sigma \Nd_{1-r_1,+\infty}^2 \, .
\end{equation}
The scheme \eqref{numibvp} is also $\ell^2$-stable with respect to boundary data, see \cite[Definition 4.5]{trefethen3}, in the sense that there exists a constant $C$ such 
that for all $\Delta t \in \, (0,1]$, the solution $(u_j^n)$ to \eqref{numibvp} with $(F_j^n)=(f_j^n)=0$ satisfies the estimate
\begin{equation*}
\sup_{n\ge 0} \, \Ng u^n \Nd_{1-r_1,+\infty}^2 \, \le \, C \, \sum_{n\ge s+1} \Delta t \, \sum_{j_1=1-r_1}^0 \, \| g_{j_1,\cdot}^n \|_{\ell^2(\Z^{d-1})}^2 \, .
\end{equation*}
\end{theorem}

\noindent Sections \ref{section2} and \ref{section3} below are devoted to the proof of Theorem \ref{mainthm}. We follow the lines of \cite{jfcX} and first explain why 
the same multiplier as in \cite{jfcX} yields an energy-dissipation balance law for the Cauchy problem (in the whole space) in the broader framework of Assumption 
\ref{assumption1}. The analysis relies on a suitable construction of the energy and dissipation functionals, which are more involved than in \cite{jfcX}. The end of 
the proof of Theorem \ref{mainthm} follows \cite{jfcX} almost word for word. We explain where the specificity of the broader framework of Assumption \ref{assumption1} 
comes into play. In an Appendix, we deal with the specific case $s=1$ (recurrence relations with two time levels) for which energy and dissipation functionals with 
\emph{local densities} can be constructed. This gives hope to later deal with finite volume space discretization techniques on unstructured meshes.

\section{The Leray-G{\aa}rding method for fully discrete Cauchy problems}
\label{section2}

This section is devoted to proving stability estimates for discretized Cauchy problems in the whole space $\Z^d$, which is the first step before considering the 
discretized initial boundary value problem \eqref{numibvp}. More precisely, we consider the simpler case of the whole space $j \in \Z^d$, and the recurrence 
relation in $\ell^2(\Z^d)$:
\begin{equation}
\label{numcauchy}
\begin{cases}
{\dps \sum_{\sigma=0}^{s+1}} Q_\sigma \, u_j^{n+\sigma} =0 \, ,& 
j\in \Z^d \, ,\quad n\ge 0 \, ,\\
u_j^n = f_j^n \, ,& j\in \Z^d \, ,\quad n=0,\dots,s \, ,
\end{cases}
\end{equation}
where the operators $Q_\sigma$ are given by \eqref{defop}. We recall that in \eqref{defop}, the $a_{\ell,\sigma}$ are real numbers and are independent of the 
small parameter $\Delta t$ (they may depend on the CFL parameters $\lambda_1,\dots,\lambda_d$), while ${\bf S}$ denotes the shift operator on the space 
grid: $({\bf S}^\ell v)_j := v_{j+\ell}$ for $j,\ell \in \Z^d$. Stability of \eqref{numcauchy} is defined as follows.

\begin{definition}[Stability for the discrete Cauchy problem]
\label{def2}
The numerical scheme defined by \eqref{numcauchy} is ($\ell^2$-) stable if $Q_{s+1}$ is an isomorphism from $\ell^2 (\Z^d)$ onto itself, and if furthermore there 
exists a constant $C_0>0$ such that for all $\Delta t \in \, (0,1]$, for all initial conditions $f^0,\dots,f^s \in \ell^2 (\Z^d)$, there holds
\begin{equation}
\label{estimcauchy}
\sup_{n \in \N} \, \Nd u^n \Nd_{-\infty,+\infty}^2 \, \le \, C_0 \, \sum_{\sigma=0}^s \, \Ng f^\sigma \Nd_{-\infty,+\infty}^2 \, .
\end{equation}
\end{definition}

\noindent Let us quickly recall, see e.g. \cite{gko}, that stability in the sense of Definition \ref{def2} is in fact independent of $\Delta t \in (0,1]$ (because 
\eqref{numcauchy} nowhere involves $\Delta t$ and the norms in \eqref{estimcauchy} can be simplified on either side by the cell volume $\prod_k \Delta x_k$), 
and can be characterized in terms of the uniform power boundedness of the so-called amplification matrix
\begin{equation}
\label{defA2pas}
{\mathcal A}(\kappa) \, := \, \begin{bmatrix}
-\widehat{Q_s}(\kappa)/\widehat{Q_{s+1}}(\kappa) & \dots & \dots & -\widehat{Q_0}(\kappa)/\widehat{Q_{s+1}}(\kappa) \\
1 & 0 & \dots & 0 \\
0 & \ddots & \ddots & \vdots \\
0 & 0 & 1 & 0 \end{bmatrix} \in {\mathcal M}_{s+1}(\C) \, ,
\end{equation}
where the $\widehat{Q_\sigma}(\kappa)$'s are defined in \eqref{dispersion} and where it is understood that ${\mathcal A}$ is defined on the largest open set of 
$\C^d$ on which $\widehat{Q_{s+1}}$ does not vanish. Let us also recall that if $Q_{s+1}$ is an isomorphism from $\ell^2(\Z^d)$ onto itself, then $\widehat{Q_{s+1}}$ 
does not vanish on $(\cercle)^d$, and therefore does not vanish on an open neighborhood of $(\cercle)^d$. With the above definition \eqref{defA2pas} for 
${\mathcal A}$, the following well-known result holds, see e.g. \cite{gko}:

\begin{proposition}[Characterization of stability for the fully discrete Cauchy problem]
\label{prop1}
Assume that $Q_{s+1}$ is an isomorphism from $\ell^2(\Z^d)$ onto itself. Then the scheme \eqref{numcauchy} is stable in the sense of Definition \ref{def2} if and 
only if there exists a constant $C_1>0$ such that the amplification matrix ${\mathcal A}$ in \eqref{defA2pas} satisfies
\begin{equation}
\label{unifpowerbound2}
\forall \, n \in \N \, ,\quad \forall \, \kappa \in (\cercle)^d \, ,\quad \left\| \, {\mathcal A}(\kappa)^n \, \right\| \, \le \, C_1 \, .
\end{equation}
In particular, the spectral radius of ${\mathcal A}(\kappa)$ should not be larger than $1$ (the so-called von Neumann condition).
\end{proposition}

The eigenvalues of ${\mathcal A}(\kappa)$ are the roots to the dispersion relation \eqref{dispersion}. When these roots are simple for all $\kappa \in (\cercle)^d$, 
the von Neumann condition is both necessary and \emph{sufficient} for stability of  \eqref{numcauchy}, see, e.g., \cite[Proposition 3]{jfcnotes}. However, Assumption 
\ref{assumption1} is more general than the situation considered in \cite{jfcX} where the roots always remain simple. Nevertheless, since the occurence of a multiple 
root only occurs in the interior $\D$ and not on the boundary $\cercle$, we easily deduce from Assumption \ref{assumption1} that the matrix ${\mathcal A}(\kappa)$ 
in \eqref{defA2pas} is \emph{geometrically regular} in the sense of \cite[Definition 3]{jfcnotes}. Hence we can still apply \cite[Proposition 3]{jfcnotes} and conclude 
that Assumption \ref{assumption1} implies stability for the Cauchy problem \eqref{numcauchy} (in the sense of Definition \ref{def2}). As in \cite{jfcX}, our goal 
now is to derive the semigroup estimate \eqref{estimcauchy} not by applying Fourier transform to \eqref{numcauchy} and using uniform power boundedness of 
${\mathcal A}$, but rather by multiplying the first equation in \eqref{numcauchy} by a suitable \emph{local} multiplier. As a warm-up, and to make things as clear 
as possible, we first deal with the simpler case where one only considers the time evolution and no additional space variable (the standard recurrence relations 
in $\C$).

\subsection{Stable recurrence relations}

In this Paragraph, we consider sequences $(v^n)_{n \in \N}$ with values in $\C$. The index $n$ should be thought of as the discrete time variable, which is the reason 
why we always write $n$ as an exponent in order to be consistent with the notation used for discretized partial differential equations. Let then $\nu \ge 1$ and let 
$a_\nu,\dots,a_0$ be some complex numbers with $a_\nu \neq 0$ (in the next Paragraphs, we choose $\nu=s+1$). It is well known that all solutions $(v^n)_{n \in \N}$ 
to the recurrence relation
$$
\forall \, n \in \N \, ,\quad a_\nu \, v^{n+\nu} \, + \, \cdots \, + \, a_0 \, v^n \, = \, 0 \, ,
$$
are bounded if and only if the polynomial:
\begin{equation}
\label{defP}
\mathbb{P}(X) \, := \, a_\nu \, X^\nu \, + \, \cdots \, + \, a_1 \, X  \, + \, a_0 \, ,
\end{equation}
has all its roots in $\Dbar$ and the roots on $\cercle$ are simple, see \cite[chapter III.3]{hnw}. This is equivalent to requiring that the companion matrix (compare with 
\eqref{defA2pas}):
$$
\begin{bmatrix}
-a_{\nu-1}/a_\nu & \dots & \dots & -a_0/a_\nu \\
1 & 0 & \dots & 0 \\
0 & \ddots & \ddots & \vdots \\
0 & 0 & 1 & 0 
\end{bmatrix} \, \in \, \mathcal{M}_\nu (\C) \, ,
$$
be power bounded. In that case, the Kreiss matrix Theorem \cite{strikwerda-wade} implies that the latter matrix is a contraction (it has a norm $\le 1$) for some Hermitian 
norm on $\C^\nu$. In \cite{jfcX}, we have obtained some explicit construction of such a Hermitian norm and an associated dissipation functional in the case where \emph{all 
the roots} of $\mathbb{P}$ in \eqref{defP} are \emph{simple} and located in $\Dbar$. The construction is based on a multiplier technique which is the discrete analogue of 
\cite[Lemme 1.1]{garding}. The inconvenience of the result in \cite{jfcX} is that even the roots in $\D$, which are associated with an exponentially decaying behavior in time, 
are assumed to be simple. We suppress this technical assumption here and explain why the multiplier technique developed in \cite{jfcX} allows to deal with the general case 
with multiple roots in $\D$.

As in \cite{jfcX}, we introduce the notation ${\bf T}$ for the shift operator in time, that is, for any sequence $(v^n)_{n \in \N}$, we define: $({\bf T}^m v)^n :=v^{n+m}$ for all 
$m,n \in \N$. The following Lemma is an extension of \cite[Lemma 1]{jfcX}.

\begin{lemma}[The energy-dissipation balance law for recurrence relations]
\label{lem1}
Let $P \in \C [X]$ be a polynomial of degree $\nu$, $\nu \ge 1$, that satisfies the following two properties:
\begin{itemize}
 \item If $P(z)=0$, then $z \in \Dbar$.
 \item If $P(z)=0$ and $z \in \cercle$, then $z$ is a simple root of $P$.
\end{itemize}
Then there exists a positive definite Hermitian form $q_e$ on $\C^\nu$, and a nonnegative Hermitian form $q_d$ on $\C^\nu$ such that for any sequence $(v^n)_{n \in \N}$ 
with values in $\C$, there holds:
\begin{multline*}
\forall \, n \in \N \, , \\
2 \, \text{\rm Re} \, \Big( \overline{{\bf T} \, (P'({\bf T}) \, v^n)} \, P({\bf T}) \, v^n \Big) \, = \, \nu \, |P({\bf T}) \, v^n|^2 \, + \, ({\bf T}-I) \, \big( q_e(v^n,\dots,v^{n+\nu-1}) \big) 
\, + \, q_d(v^n,\dots,v^{n+\nu-1}) \, .
\end{multline*}
In particular, for any sequence $(v^n)_{n \in \N}$ that satisfies the recurrence relation
\begin{equation*}
\forall \, n \in \N \, ,\quad P({\bf T}) \, v^n \, = \, 0 \, ,
\end{equation*}
the sequence $(q_e(v^n,\dots,v^{n+\nu-1}))_{n\in \N}$ is nonincreasing.
\end{lemma}

\noindent The multiplier ${\bf T} \, P'({\bf T}) \, v^n$ used in Lemma \ref{lem1} is the same as in \cite{jfcX}. We shall see below in the proof why the expressions provided 
in \cite{jfcX} for the energy and dissipation functions $q_e,q_d$ can not cover the case of multiple roots and how they should be modified.

\begin{proof}
Let us first recall the proof in \cite{jfcX} in the case of \emph{simple} roots because this is the starting point for the general case we consider here. We therefore assume 
for now that $P$ has degree $\nu$ and only has simple roots $z_1,\dots,z_\nu$ located in $\Dbar$. We write
$$
P(X) \, = \, a \, \prod_{j=1}^\nu \, (X \, - \, z_j) \, ,
$$
with $a \neq 0$, and introduce the Lagrange polynomials:
$$
\forall \, k \, = \, 1,\dots,\nu \, ,\quad P_k(X) \, := \, a \, \prod_{\substack{j=1 \\ j \neq k}}^\nu \, (X \, - \, z_j) \, .
$$
Since the $z_j$'s are pairwise distinct, the $P_k$'s form a basis of $\C_{\nu-1}[X]$. Moreover, the following relation was obtained in \cite{jfcX}:
\begin{equation}
\label{enerdissip1}
2 \, \text{\rm Re} \, \Big( \overline{{\bf T} \, (P'({\bf T}) \, v^n)} \, P({\bf T}) \, v^n \Big) \, - \, \nu \, |P({\bf T}) \, v^n|^2 
\, = \, ({\bf T}-I) \, \left\{ \, \sum_{k=1}^\nu \, |P_k({\bf T}) \, v^n|^2 \, \right\} \, + \, \sum_{k=1}^\nu \, \big( 1 \, - \, |z_k|^2 \big) \, |P_k({\bf T}) \, v^n|^2 \, .
\end{equation}
The conclusion of Lemma \ref{lem1} is then obtained by introducing the energy ($q_e$) and dissipation ($q_d$) forms:
\begin{align}
\forall \, (w^0,\dots,w^{\nu-1}) \in \C^\nu \, ,\quad q_e(w^0,\dots,w^{\nu-1}) & \, := \, \sum_{k=1}^\nu \, |P_k({\bf T}) \, w^0|^2 \, ,\label{defqe}\\
q_d(w^0,\dots,w^{\nu-1}) & \, := \, \sum_{k=1}^\nu \, (1 \, - \, |z_k|^2) \, |P_k({\bf T}) \, w^0|^2 \, .\label{defqd}
\end{align}
When the roots of $P$ are located in $\Dbar$, $q_d$ is obviously nonnegative (this property does not depend on the fact that the roots are simple). When furthermore the 
roots of $P$ are simple, the $P_k$'s form a basis of $\C_{\nu-1}[X]$ and $q_e$ is positive definite. The conclusion follows.
\bigskip

We now turn to the general case and therefore no longer assume that the roots of $P$ in $\D$ are simple. For the sake of clarity, we label the pairwise distinct roots of $P$ 
as $z_1,\dots,z_m$ and let $\mu_1,\dots,\mu_m$ denote the corresponding multiplicities. We thus have:
$$
P(X) \, = \, a \, \prod_{j=1}^m \, (X \, - \, z_j)^{\mu_j} \, ,
$$
for some $a \neq 0$, and we introduce the polynomials:
$$
\forall \, k \, = \, 1,\dots,m \, ,\quad P_k(X) \, := \, a \, (X \, - \, z_k)^{\mu_k-1} \, \prod_{\substack{j=1 \\ j \neq k}}^m \, (X \, - \, z_j)^{\mu_j} \, .
$$
We thus get the relation:
$$
P' \, = \, \sum_{k=1}^m \, \mu_k \, P_k \, ,
$$
and it is a simple exercise to adapt the computation in \cite{jfcX} to obtain the relation (compare with \eqref{enerdissip1}):
\begin{multline}
\label{enerdissip2}
2 \, \text{\rm Re} \, \Big( \overline{{\bf T} \, (P'({\bf T}) \, v^n)} \, P({\bf T}) \, v^n \Big) \, - \, \nu \, |P({\bf T}) \, v^n|^2 
\, = \, ({\bf T}-I) \, \left\{ \, \sum_{k=1}^m \, \mu_k \, |P_k({\bf T}) \, v^n|^2 \, \right\} \\
\, + \, \sum_{k=1}^m \, \mu_k \, (1 \, - \, |z_k|^2) \, |P_k({\bf T}) \, v^n|^2 \, .
\end{multline}
The problem which we are facing is that there are too few polynomials $P_k$ to span the whole space $\C_{\nu-1}[X]$. The trick consists in adding to the energy part on 
the right hand side of \eqref{enerdissip2} some nonnegative Hermitian forms in order to gain positive definiteness, while still keeping the corresponding dissipation form 
nonnegative. This ``add and subtract'' trick is performed below.
\bigskip

As long as a root $z_k$ is at least double ($\mu_k \ge 2$), we introduce the polynomials:
$$
\forall \, j \, = \, 1,\dots,\mu_k-1 \, ,\quad Q_{k,j}(X) \, := \, a \, (X \, - \, z_k)^{j-1} \, \prod_{\substack{ \ell=1 \\ \ell \neq k}}^m \, (X \, - \, z_\ell)^{\mu_\ell} \, ,
$$
each of which being of degree $\le \nu-2$. (Later we shall use the fact that $X \, Q_{k,j}(X)$ has degree $\le \nu-1$.) We go back to \eqref{enerdissip2} 
and add/subtract suitable quantities as follows:
\begin{align}
2 \, \text{\rm Re} \, \Big( \overline{{\bf T} \, (P'({\bf T}) \, v^n)} \, & \, P({\bf T}) \, v^n \Big) \, - \, \nu \, |P({\bf T}) \, v^n|^2 \notag \\
&= \, ({\bf T}-I) \, \left\{ \, \sum_{k=1}^m \, \mu_k \, |P_k({\bf T}) \, v^n|^2 
\, + \, \sum_{k=1}^m \, \sum_{j=1}^{\mu_k-1} \, \varepsilon^{\mu_k-j} \, (1 \, - \, |z_k|^2)^{2(\mu_k-j)} \, |Q_{k,j}({\bf T}) \, v^n|^2 \, \right\} \label{enerdissip3} \\
&\quad + \, \sum_{k=1}^m \, \mu_k \, (1 \, - \, |z_k|^2) \, |P_k({\bf T}) \, v^n|^2 \notag \\
&\quad + \, \sum_{k=1}^m \, \sum_{j=1}^{\mu_k-1} \, \varepsilon^{\mu_k-j} \, (1 \, - \, |z_k|^2)^{2(\mu_k-j)} \, 
\left( \, |Q_{k,j}({\bf T}) \, v^n|^2 \, - \, |Q_{k,j}({\bf T}) \, v^{n+1}|^2 \, \right) \, ,\notag
\end{align}
where $\varepsilon>0$ is a parameter to be fixed later on (any choice $0<\varepsilon \le 1/4$ will do). In \eqref{enerdissip3}, it is understood that if $\mu_k=1$ 
(that is, if the root $z_k$ is simple), then we do not add any polynomial $Q_{k,j}$, the range of indices $1 \le j \le \mu_k-1$ being empty. Moreover, we recall 
that if $\mu_k \ge 2$ for some $k$, then we have $|z_k|<1$ so the coefficient of the Hermitian form $|Q_{k,j}({\bf T}) \, w^0|^2$ on the second line of 
\eqref{enerdissip3} will be positive.

It remains to show that for some suitably chosen parameter $\varepsilon>0$, the decomposition \eqref{enerdissip3} yields the result of Lemma \ref{lem1}. 
Let us first observe that the $\nu$ polynomials
$$
Q_{1,1} \, , \, \dots \, , \, Q_{1,\mu_1-1} \, , \, P_1 \, , \, \dots \, , \, Q_{m,1} \, , \, \dots \, , \, Q_{m,\mu_m-1} \, , \, P_m \, ,
$$
span the space $\C_{\nu-1}[X]$ (this is nothing but the classical Hermite interpolation problem). Since the quantity $1-|z_k|^2$ is positive as long as $\mu_k$ is 
larger than $2$, any choice $\varepsilon>0$ will make the Hermitian form $q_e$ defined on $\C^\nu$ by:
\begin{multline}
\label{defqemult}
\forall \, (w^0,\dots,w^{\nu-1}) \in \C^\nu \, ,\\
q_e(w^0,\dots,w^{\nu-1}) \, := \, \sum_{k=1}^m \, \mu_k \, |P_k({\bf T}) \, w^0|^2 
\, + \, \sum_{k=1}^m \, \sum_{j=1}^{\mu_k-1} \, \varepsilon^{\mu_k-j} \, (1 \, - \, |z_k|^2)^{2(\mu_k-j)} \, |Q_{k,j}({\bf T}) \, w^0|^2 \, ,
\end{multline}
positive definite. We thus now define a Hermitian form $q_d$ on $\C^\nu$ by:
\begin{multline}
\label{defqdmult}
\forall \, (w^0,\dots,w^{\nu-1}) \in \C^\nu \, ,\quad q_d(w^0,\dots,w^{\nu-1}) \, := \, \sum_{k=1}^m \, \mu_k \, (1 \, - \, |z_k|^2) \, |P_k({\bf T}) \, w^0|^2 \\
\, + \, \sum_{k=1}^m \, \sum_{j=1}^{\mu_k-1} \, \varepsilon^{\mu_k-j} \, (1 \, - \, |z_k|^2)^{2(\mu_k-j)} \, \left( \, |Q_{k,j}({\bf T}) \, w^0|^2 \, - \, |Q_{k,j}({\bf T}) \, w^1|^2 \, \right) \, ,
\end{multline}
and we are going to show that a convenient choice of $\varepsilon$ makes $q_d$ nonnegative. (Let us observe here that it is crucial to have the degree of 
$Q_{k,j}$ less than $\nu-2$ so that the quantity $Q_{k,j}({\bf T}) \, w^1$ is a linear combination of $w^1,\dots,w^{\nu-1}$.) With the above definitions \eqref{defqemult} 
and \eqref{defqdmult} for $q_e$ and $q_d$, the energy balance law \eqref{enerdissip3} reads as claimed in Lemma \ref{lem1}, so the only remaining task 
is to show that $q_d$ is nonnegative for a convenient choice of $\varepsilon>0$.

We use below the convention $Q_{k,\mu_k} := P_k$, which is compatible with the above definition of $P_k$ and of the $Q_{k,j}$'s. Observing that there holds:
$$
\forall \, j \, = \, 1,\dots,\mu_k-1 \, ,\quad X \, Q_{k,j}(X) \, = \, Q_{k,j+1} \, + \, z_k \, Q_{k,j} \, ,
$$
we have for any $k=1,\dots,m$:
\begin{align*}
\sum_{j=1}^{\mu_k-1} \, \varepsilon^{\mu_k-j} \, & \, (1 \, - \, |z_k|^2)^{2(\mu_k-j)} \, \Big( \, |Q_{k,j}({\bf T}) \, w^0|^2 -|Q_{k,j}({\bf T}) \, w^1|^2 \, \Big) \\
&= \, \sum_{j=1}^{\mu_k-1} \, \varepsilon^{\mu_k-j} \, (1 \, - \, |z_k|^2)^{2(\mu_k-j)} \, \Big( \, 
|Q_{k,j}({\bf T}) \, w^0|^2 \, - \, |Q_{k,j+1}({\bf T}) \, w^0  \, + \, z_k \, Q_{k,j}({\bf T}) \, w^0|^2 \, \Big) \\
&= \, \sum_{j=1}^{\mu_k-1} \, \varepsilon^{\mu_k-j} \, (1 \, - \, |z_k|^2)^{2(\mu_k-j)} \, \Big( \, (1 \, - \, |z_k|^2) \, |Q_{k,j}({\bf T}) \, w^0|^2 \, - \, |Q_{k,j+1}({\bf T}) \, w^0|^2 \Big) \\
&\quad -\sum_{j=1}^{\mu_k-1} \, \varepsilon^{\mu_k-j} \, (1 \, - \, |z_k|^2)^{2(\mu_k-j)} \, 2 \, \text{\rm Re} \, \Big( \overline{z_k \, Q_{k,j}({\bf T}) \, w^0} \, Q_{k,j+1}({\bf T}) \, w^0 \Big) \, .
\end{align*}
We use Young's inequality as follows:
$$
\left| 2 \, \text{\rm Re} \, \Big( \overline{z_k \, Q_{k,j}({\bf T}) \, w^0} \, Q_{k,j+1}({\bf T}) \, w^0 \Big) \right| \, \le \, 
\dfrac{1}{2} \, (1 \, - \, |z_k|^2) \, |Q_{k,j}({\bf T}) \, w^0|^2 \, + \, \dfrac{2 \, |z_k|^2}{1 \, - \, |z_k|^2} \, |Q_{k,j+1}({\bf T}) \, w^0|^2 \, ,
$$
and thus derive the lower bound:
\begin{align*}
\sum_{j=1}^{\mu_k-1} \, \varepsilon^{\mu_k-j} \, (1 \, - \, |z_k|^2)^{2(\mu_k-j)} \, \Big( \, |Q_{k,j}({\bf T}) \, w^0|^2 &\, - \, |Q_{k,j}({\bf T}) \, w^1|^2 \, \Big) \\
&\ge \, \sum_{j=1}^{\mu_k-1} \, \dfrac{1}{2} \, \varepsilon^{\mu_k-j} \, (1 \, - \, |z_k|^2)^{2(\mu_k-j)+1} \, |Q_{k,j}({\bf T}) \, w^0|^2 \\
&\quad -\sum_{j=1}^{\mu_k-1} \, \varepsilon^{\mu_k-j} \, (1 \, - \, |z_k|^2)^{2(\mu_k-j)} \, \dfrac{1 \, + \, |z_k|^2}{1 \, - \, |z_k|^2} \, |Q_{k,j+1}({\bf T}) \, w^0|^2 \, .
\end{align*}
Shifting indices, we get:
\begin{align*}
\sum_{j=1}^{\mu_k-1} \, \varepsilon^{\mu_k-j} \, (1 \, - \, |z_k|^2)^{2(\mu_k-j)} \, \Big( \, |Q_{k,j}({\bf T}) \, w^0|^2 &\, - \, |Q_{k,j}({\bf T}) \, w^1|^2 \, \Big) \\
&\ge \, \sum_{j=0}^{\mu_k-2} \, \dfrac{1}{2 \, \varepsilon} \, \varepsilon^{\mu_k-j} \, (1 \, - \, |z_k|^2)^{2(\mu_k-j)-1} \, |Q_{k,j+1}({\bf T}) \, w^0|^2 \\
&\quad -\sum_{j=1}^{\mu_k-1} \, \varepsilon^{\mu_k-j} \, (1 \, - \, |z_k|^2)^{2(\mu_k-j)} \, \dfrac{1 \, + \, |z_k|^2}{1 \, - \, |z_k|^2} \, |Q_{k,j+1}({\bf T}) \, w^0|^2 \, .
\end{align*}
Restricting from now on to $0<\varepsilon \le 1/4$, we have $1/(2 \, \varepsilon) \ge 2 \ge 1 \, + \, |z_k|^2$ and all terms corresponding to the indices $j=1,\dots,\mu_k-2$ 
in the above two sums match to give a nonnegative quantity (the first one for $j=0$ obviously gives a nonnegative contribution since it only appears in the first sum). 
Hence we can keep only the very last term corresponding to $j=\mu_k-1$ and we have thus derived the lower bound:
\begin{multline*}
\sum_{j=1}^{\mu_k-1} \, \varepsilon^{\mu_k-j} \, (1 \, - \, |z_k|^2)^{2(\mu_k-j)} \, \Big( \, |Q_{k,j}({\bf T}) \, w^0|^2 \, - \, |Q_{k,j}({\bf T}) \, w^1|^2 \, \Big) \\
\ge -\varepsilon \, (1 \, - \, |z_k|^2) \, (1 \, + \, |z_k|^2) \, |Q_{k,\mu_k}({\bf T}) \, w^0|^2 \ge -\dfrac{1}{2} \, (1 \, - \, |z_k|^2) \, |P_k({\bf T}) \, w^0|^2 \, ,
\end{multline*}
where we have used $|z_k| \le 1$ and $\varepsilon \le 1/4$ in the last inequality. Going back to the definition \eqref{defqdmult} of $q_d$, and summing over the $k$'s, 
we obtain that the Hermitian form $q_d$ is nonnegative for any choice of $\varepsilon$ within the interval $(0,1/4]$. The proof of Lemma \ref{lem1} is complete.
\end{proof}

\subsection{The energy-dissipation balance for finite difference schemes}

In this Paragraph, we consider the numerical scheme \eqref{numcauchy}. We introduce the following notation: 
\begin{equation}
\label{defLM}
L \, := \, \sum_{\sigma=0}^{s+1} \, {\bf T}^\sigma \, Q_\sigma \, ,\quad M \, := \, \sum_{\sigma=0}^{s+1} \, \sigma \, {\bf T}^\sigma \, Q_\sigma \, .
\end{equation}
Thanks to Fourier analysis, the following result will be a consequence of Lemma \ref{lem1}.

\begin{proposition}[The energy-dissipation balance law for finite difference schemes]
\label{prop2}
Let Assumptions \ref{assumption0} and \ref{assumption1} be satisfied. Then there exist a continuous coercive quadratic form $E$ and a continuous nonnegative 
quadratic form $D$ on $\ell^2(\Z^d;\R)^{s+1}$ such that for all sequences $(v^n)_{n \in \N}$ with values in $\ell^2(\Z^d;\R)$ and for all $n \in \N$, there holds
\begin{equation*}
2 \, \langle \, M \, v^n,L\, v^n \, \rangle_{-\infty,+\infty} \, = \, (s+1) \, \Ng \, L \, v^n \, \Nd_{-\infty,+\infty}^2 \, + \, ({\bf T}-I) \, E(v^n,\dots,v^{n+s}) 
\, + \, D(v^n,\dots,v^{n+s}) \, .
\end{equation*}
In particular, for any choice of initial data $f^0,\dots,f^s \in \ell^2(\Z^d;\R)$, the solution to \eqref{numcauchy} satisfies
\begin{equation*}
\sup_{n \in \N} \, E(u^n,\dots,u^{n+s}) \, \le \, E(f^0,\dots,f^s) \, ,
\end{equation*}
and \eqref{numcauchy} is ($\ell^2$-)stable.
\end{proposition}

\begin{proof}
We use the same notation $v^n$ for the sequence $(v_j^n)_{j \in \Z^d}$ and the corresponding step function on $\R^d$ whose value on the cell 
$[j_1 \, \Delta x_1,(j_1+1) \, \Delta x_1) \times \cdots \times [j_d \, \Delta x_d,(j_d+1) \, \Delta x_d)$ equals $v_j^n$ for any $j \in \Z^d$. Then Plancherel's 
Theorem gives the identity
\begin{multline}
\label{egaliteLG}
2 \, \langle \, M \, v^n,L\, v^n \, \rangle_{-\infty,+\infty} \, - \, (s+1) \, \Ng \, L \, v^n \, \Nd_{-\infty,+\infty}^2 \\
\, = \, \int_{\R^d} 2 \, \text{\rm Re} \Big( \overline{{\bf T} \, (P_\kappa'({\bf T}) \, \widehat{v^n}(\xi))} \, P_\kappa({\bf T}) \, \widehat{v^n} (\xi) \Big) 
\, - \, (s+1) \, \big| P_\kappa({\bf T}) \, \widehat{v^n}(\xi) \big|^2 \, \dfrac{{\rm d}\xi}{(2\, \pi)^d} \, ,
\end{multline}
where $\widehat{v^n}$ denotes the Fourier transform (in $L^2(\R^d)$) of the function $v^n$, and where we have let
\begin{equation*}
P_\kappa (z) \, := \, \sum_{\sigma=0}^{s+1} \, \widehat{Q_\sigma} \big( \kappa_1,\dots,\kappa_d \big) \, z^\sigma \, ,\quad 
\kappa_j \, := \, {\rm e}^{i \, \xi_j \, \Delta x_j} \in \cercle \, ,
\end{equation*}
and $P'_\kappa(z)$ denotes the derivative of $P_\kappa$ with respect to $z$.

The construction of the quadratic forms $E$ and $D$ is made, as in \cite{jfcX}, of the superposition of appropriate energy and dissipation Hermitian forms 
for each frequency $\kappa \in (\cercle)^d$, each coordinate $\kappa_j$ being a placeholder for $\exp (i \, \xi_j \, \Delta x_j)$. Here, unlike \cite{jfcX}, the 
polynomial $P_\kappa$ either only has simple roots in $\Dbar$ or it has one multiple root in $\D$ and all other roots are simple. We cannot therefore 
construct the energy and dissipation forms in a unified manner. Below we shall use the analysis of Lemma \ref{lem1} in the neighborhood of finitely many 
points in $(\cercle)^d$ where $P_\kappa$ has a multiple root and we shall use \cite[Lemma 1]{jfcX} in the neighborhood of all points where $P_\kappa$ 
only has simple roots. (This is the reason why we have recalled the proof of Lemma \ref{lem1} in the case where all roots are simple.) We shall eventually 
glue things together thanks to a suitable partition of unity.

Let us first consider the point $\underline{\kappa}^{(1)} \in (\cercle)^d$ for which $P_{\underline{\kappa}^{(1)}}$ has one multiple root (of multiplicity $m_1$) 
in $\D$ and in the neighborhood of which we have a smooth splitting of the eigenmodes $z_1,\dots,z_{m_1}$. The other roots $z_{m_1+1},\dots,z_{s+1}$ 
are simple and can thus be determined holomorphically with respect to $\kappa$ in the neighborhood of $\underline{\kappa}^{(1)}$. Keeping in mind that 
the dominant coefficient of the polynomial $P_\kappa (z)$ equals $\widehat{Q_{s+1}}(\kappa)$ (which is nonzero for $\kappa \in (\cercle)^d$), we consider 
some $\kappa \in (\cercle)^d$ sufficiently close to $\underline{\kappa}^{(1)}$ and introduce the Lagrange polynomials:
$$
\forall \, k \, = \, 1,\dots,s+1 \, ,\quad P_{k,\kappa}(z) \, := \, \widehat{Q_{s+1}}(\kappa) \, \prod_{\substack{j=1 \\ j \neq k}}^{s+1} \, \big( z \, - \, z_j(\kappa) \big) \, .
$$
We then introduce the following energy and dissipation Hermitian forms on $\C^{s+1}$ (below, $\kappa$ always denotes an element of $(\cercle)^d$ that is 
sufficiently close to $\underline{\kappa}^{(1)}$ so that all considered quantities are well-defined):
\begin{align}
\forall \, (w^0,\dots,w^s) \in \C^{s+1} \, ,\quad & \notag \\
q_{e,\kappa}(w^0,\dots,w^s) \, :=& \, \sum_{k=1}^{s+1} \, |P_{k,\kappa}({\bf T}) \, w^0|^2 \label{defqemult-edp} \\
+ \, \sum_{k=1}^{m_1} \, \sum_{j=1}^{m_1-1} \, \varepsilon^{m_1-j} \, & \, \big( 1 \, - \, |z_k(\kappa)|^2 \big)^{2(m_1-j)} \, 
\left| \widehat{Q_{s+1}}(\kappa) \, \big( {\bf T}-z_k(\kappa) \big)^{j-1} \prod_{\ell=m_1+1}^{s+1} \! \big( {\bf T} \, - \, z_\ell(\kappa) \big) \, w^0 \right|^2 ,\notag \\
q_{d,\kappa}(w^0,\dots,w^s) \, :=& \, \sum_{k=1}^{s+1} \, \big( 1 \, - \, |z_k(\kappa)|^2 \big) \, |P_{k,\kappa}({\bf T}) \, w^0|^2 \label{defqdmult-edp} \\
+ \, \sum_{k=1}^{m_1} \, \sum_{j=1}^{m_1-1} \, \varepsilon^{m_1-j} \, & \, \big( 1 \, - \, |z_k(\kappa)|^2 \big)^{2(m_1-j)} \, \times \notag \\
\Big\{ \Big| \widehat{Q_{s+1}}(\kappa) \, \big( {\bf T}-z_k(\kappa) \big)^{j-1} \! & \prod_{\ell=m_1+1}^{s+1} \! \big( {\bf T} - z_\ell(\kappa) \big) \, w^0 \Big|^2 
-\Big| \widehat{Q_{s+1}}(\kappa) \, \big( {\bf T}-z_k(\kappa) \big)^{j-1} \prod_{\ell=m_1+1}^{s+1} \! \big( {\bf T} - z_\ell(\kappa) \big) \, w^1 \Big|^2 \Big\} \, ,\notag
\end{align}
where $\varepsilon>0$ is a parameter to be fixed later on. Using the decomposition \eqref{enerdissip1} which we have recalled in the proof of 
Lemma \ref{lem1}, we have the decomposition 
\begin{equation}
\label{relationLG}
2 \, \text{\rm Re} \Big( \overline{{\bf T} \, (P_\kappa'({\bf T}) \, w^0)} \, P_\kappa({\bf T}) \, w^0 \Big) \, - \, (s+1) \, |P_\kappa({\bf T}) \, w^0|^2 
\, = \, ({\bf T}-I) \, (q_{e,\kappa}(w^0,\dots,w^s)) \, + \, q_{d,\kappa}(w^0,\dots,w^s) \, ,
\end{equation}
for all vectors $(w^0,\dots,w^s) \in \C^{s+1}$, because we have just added and subtracted some Hermitian forms to the energy-dissipation balance law 
\eqref{enerdissip1}. It remains to prove that $q_{d,\kappa}$ in \eqref{defqdmult-edp} is nonnegative and that $q_{e,\kappa}$ in \eqref{defqemult-edp} is 
positive definite. Let us start with $q_{e,\kappa}$. If $\kappa$ does not equal $\underline{\kappa}^{(1)}$, we know from Assumption \ref{assumption1} 
that the roots $z_1(\kappa),\dots,z_{s+1}(\kappa)$ are pairwise distinct so the Lagrange polynomials $P_{k,\kappa}$ form a basis of $\C_s[X]$. Hence 
$q_{e,\kappa}$ in \eqref{defqemult-edp} is positive definite because we have added a nonnegative form to a positive definite one. We thus now consider 
the case $\kappa=\underline{\kappa}^{(1)}$ for which the $m_1$ first roots $z_1,\dots,z_{m_1}$ all collapse to $\underline{z}^{(1)}$ and the $m_1$ first 
Lagrange polynomials $P_{1,\underline{\kappa}^{(1)}},\dots,P_{m_1,\underline{\kappa}^{(1)}}$ are all equal. At the base point $\kappa = \underline{\kappa}^{(1)}$, 
the definition \eqref{defqemult-edp} thus reduces to:
\begin{multline*}
q_{e,\underline{\kappa}^{(1)}}(w^0,\dots,w^s) \, = \, m_1 \, |P_{1,\underline{\kappa}^{(1)}}({\bf T}) \, w^0|^2 
\, + \, \sum_{k=m_1+1}^{s+1} \, |P_{k,\underline{\kappa}^{(1)}}({\bf T}) \, w^0|^2 \\
\, + \, m_1 \, \sum_{j=1}^{m_1-1} \, \varepsilon^{m_1-j} \, \big( 1 \, - \, |\underline{z}^{(1)}|^2 \big)^{2(m_1-j)} \, \left| \widehat{Q_{s+1}}(\underline{\kappa}^{(1)}) \, 
\big( {\bf T}-\underline{z}^{(1)} \big)^{j-1} \, \prod_{\ell=m_1+1}^{s+1} \, \big( {\bf T} \, - \, z_\ell(\underline{\kappa}^{(1)}) \big) \, w^0 \right|^2 \, ,
\end{multline*}
which (up to the harmless positive constant $m_1$ in the second line) coincides with our definition of the Hermitian form in \eqref{defqemult}. Since the 
polynomials:
\begin{multline*}
P_{1,\underline{\kappa}^{(1)}}(X) \, , \, P_{m_1+1,\underline{\kappa}^{(1)}}(X) \, , \, \dots \, , \, P_{s+1,\underline{\kappa}^{(1)}}(X) \, , \, 
\widehat{Q_{s+1}}(\underline{\kappa}^{(1)}) \, \prod_{\ell=m_1+1}^{s+1} \, \big( X \, - \, z_\ell(\underline{\kappa}^{(1)}) \big) \, , \\
\widehat{Q_{s+1}}(\underline{\kappa}^{(1)}) \, \big( X-\underline{z}^{(1)} \big) \, \prod_{\ell=m_1+1}^{s+1} \, \big( X \, - \, z_\ell(\underline{\kappa}^{(1)}) \big) \, , \, 
\dots \, , \, 
\widehat{Q_{s+1}}(\underline{\kappa}^{(1)}) \, \big( X-\underline{z}^{(1)} \big)^{m_1-2} \, \prod_{\ell=m_1+1}^{s+1} \, \big( X \, - \, z_\ell(\underline{\kappa}^{(1)}) \big) \, ,
\end{multline*}
form a basis of $\C_s[X]$ (this is again the classical Hermite interpolation problem), the form $q_{e,\underline{\kappa}^{(1)}}$ is positive definite as long as 
the parameter $\varepsilon$ is a fixed positive constant (the choice $\varepsilon=1/8$ that is made below will do). Moreover, once $\varepsilon$ is fixed, the 
form $q_{e,\kappa}$ depends in a $\mathcal{C}^\infty$ way on $\kappa$ in the neighborhood of $\underline{\kappa}^{(1)}$.
\bigskip

We now show that the form $q_{d,\kappa}$ in \eqref{defqdmult-edp} is nonnegative for a well-chosen parameter $\varepsilon>0$ and $\kappa \in (\cercle)^d$ 
sufficiently close to $\underline{\kappa}^{(1)}$. The argument is quite similar to what we have done in the proof of Lemma \ref{lem1} but we now need to take 
into account that the $m_1$ first eigenmodes $z_1,\dots,z_{m_1}$ split for $\kappa \neq \underline{\kappa}^{(1)}$, which will make us choose $\varepsilon>0$ 
slightly smaller than in the proof of Lemma \ref{lem1} in order to absorb an additional error. Before going on, let us recall that the eigenmodes $z_1(\kappa), 
\dots, z_{s+1}(\kappa)$ belong to $\Dbar$ for $\kappa \in (\cercle)^d$ close to $\underline{\kappa}^{(1)}$ with $\kappa \neq \underline{\kappa}^{(1)}$. By 
continuity, this implies that they also belong to $\Dbar$ for $\kappa=\underline{\kappa}^{(1)}$. Hereafter, we shall consider $\kappa \in (\cercle)^d$ close 
to $\underline{\kappa}^{(1)}$ and shall therefore feel free to use the inequality $|z_\ell(\kappa)| \le 1$ for all $\ell=1,\dots,s+1$ (the so-called von Neumann 
condition).

Let us consider some vector $(w^0,\dots,w^s) \in \C^{s+1}$ and let us introduce the notation:
\begin{equation}
\label{defWkj}
\forall \, k,j \, = \, 1,\dots,m_1 \, ,\quad W_{k,j} \, := \, \widehat{Q_{s+1}}(\kappa) \, \big( {\bf T}-z_k(\kappa) \big)^{j-1} \, 
\prod_{\ell=m_1+1}^{s+1} \, \big( {\bf T} \, - \, z_\ell(\kappa) \big) \, w^0 \, ,
\end{equation}
where the complex numbers $W_{k,j}$ (which, according to \eqref{defWkj}, are linear combinations of $w^0,\dots,w^s$) also depend on $\kappa$ but 
there is no need to keep track of this in what follows. We start from the definition \eqref{defqdmult-edp} and derive the lower bound:
\begin{multline*}
q_{d,\kappa}(w^0,\dots,w^s) \, \ge \, \sum_{k=1}^{m_1} \, \big( 1 \, - \, |z_k(\kappa)|^2 \big) \, |P_{k,\kappa}({\bf T}) \, w^0|^2 \\
+ \, \sum_{k=1}^{m_1} \, \sum_{j=1}^{m_1-1} \, \varepsilon^{m_1-j} \, \big( 1 \, - \, |z_k(\kappa)|^2 \big)^{2(m_1-j)} \, 
\Big( |W_{k,j}|^2 -|W_{k,j+1} \, + \, z_k(\kappa) \, W_{k,j}|^2 \Big) \, .
\end{multline*}
Expanding the square modulus $|W_{k,j+1} \, + \, z_k(\kappa) \, W_{k,j}|^2$ and using Young's inequality under the form:
\begin{align*}
\left| 2 \, \text{\rm Re} \, \Big( \overline{z_k(\kappa) \, W_{k,j}} \, W_{k,j+1} \Big) \right| \, 
&\le \, \dfrac{1}{2} \, \big( 1 \, - \, |z_k(\kappa)|^2 \big) \, |W_{k,j}|^2 \, + \, \dfrac{2 \, |z_k(\kappa)|^2}{1 \, - \, |z_k(\kappa)|^2} \, |W_{k,j+1}|^2 \\
&\le \, \dfrac{1}{2} \, \big( 1 \, - \, |z_k(\kappa)|^2 \big) \, |W_{k,j}|^2 \, + \, \dfrac{1+|z_k(\kappa)|^2}{1 \, - \, |z_k(\kappa)|^2} \, |W_{k,j+1}|^2 \, ,
\end{align*}
we get:
\begin{multline*}
q_{d,\kappa}(w^0,\dots,w^s) \, \ge \, \sum_{k=1}^{m_1} \, \big( 1 \, - \, |z_k(\kappa)|^2 \big) \, |P_{k,\kappa}({\bf T}) \, w^0|^2 \\
+ \, \sum_{k=1}^{m_1} \, \sum_{j=1}^{m_1-1} \, \varepsilon^{m_1-j} \, \big( 1 \, - \, |z_k(\kappa)|^2 \big)^{2(m_1-j)} \, 
\left( \dfrac{1 \, - \, |z_k(\kappa)|^2}{2} \, |W_{k,j}|^2 -\dfrac{2}{1 \, - \, |z_k(\kappa)|^2} \, |W_{k,j+1}|^2 \right) \, .
\end{multline*}
After shifting indices, we end up with:
\begin{multline*}
q_{d,\kappa}(w^0,\dots,w^s) \, \ge \, \sum_{k=1}^{m_1} \, \big( 1 \, - \, |z_k(\kappa)|^2 \big) \, |P_{k,\kappa}({\bf T}) \, w^0|^2 \\
+ \, \sum_{k=1}^{m_1} \, \sum_{j=1}^{m_1-1} \, \dfrac{\varepsilon^{m_1-j}}{2} \, \big( 1 \, - \, |z_k(\kappa)|^2 \big)^{2(m_1-j)+1} \, |W_{k,j}|^2 
\, - \, \sum_{k=1}^{m_1} \, \sum_{j=2}^{m_1} \, 2 \, \varepsilon \, \varepsilon^{m_1-j} \, \big( 1 \, - \, |z_k(\kappa)|^2 \big)^{2(m_1-j)+1} \, |W_{k,j}|^2 \, .
\end{multline*}
Instead of choosing $\varepsilon \in (0,1/4]$ as in the proof of Lemma \ref{lem1}, we make  the more restrictive choice $\varepsilon \in (0,1/8]$ and 
thus obtain:
\begin{multline}
\label{inegfinale-edp}
q_{d,\kappa}(w^0,\dots,w^s) \, \ge \, \sum_{k=1}^{m_1} \, \big( 1 \, - \, |z_k(\kappa)|^2 \big) \, |P_{k,\kappa}({\bf T}) \, w^0|^2 
- \dfrac{\big( 1 \, - \, |z_k(\kappa)|^2 \big)}{4} \, |W_{k,m_1}|^2 \\
+ \, \sum_{k=1}^{m_1} \, \sum_{j=1}^{m_1-1} \, \dfrac{\varepsilon^{m_1-j}}{4} \, \big( 1 \, - \, |z_k(\kappa)|^2 \big)^{2(m_1-j)+1} \, |W_{k,j}|^2 \, .
\end{multline}
We go back to the definition of the Lagrange polynomial $P_{k,\kappa}$ and of the complex numbers $W_{k,j}$. For $k=1,\dots,m_1$, we have:
$$
P_{k,\kappa}(z) \, = \, \widehat{Q_{s+1}}(\kappa) \, \prod_{\substack{j=1 \\ j \neq k}}^{m_1} \, \big( z \, - \, z_j(\kappa) \big) 
\,  \prod_{\ell=m_1+1}^{s+1} \, \big( z \, - \, z_\ell(\kappa) \big) \, .
$$
The goal is to absorb in \eqref{inegfinale-edp} the only negative term by means of all other positive quantities. To do this, we observe that we 
can expand the polynomial
$$
\big( X \, - \, z_k(\kappa) \big)^{m_1-1} \, ,
$$
on the basis of $\C_{m_1-1}[X]$ formed by the polynomials :
$$
1 \, , \, \big( X \, - \, z_k(\kappa) \big) \, , \, \big( X \, - \, z_k(\kappa) \big)^{m_1-2} \, , \, \prod_{\substack{j=1 \\ j \neq k}}^{m_1} \, \big( X \, - \, z_j(\kappa) \big) \, .
$$
The linear system for determining the coefficients is lower triangular and has determinant $1$ so we can write for each $k=1,\dots,m_1$:
\begin{equation}
\label{decompositionpoly}
\big( X \, - \, z_k(\kappa) \big)^{m_1-1} \, = \, \prod_{\substack{j=1 \\ j \neq k}}^{m_1} \, \big( X \, - \, z_j(\kappa) \big) \, 
+ \sum_{j=1}^{m_1-1} \, a_{k,j}(\kappa) \, \big( X \, - \, z_k(\kappa) \big)^{j-1} \, , 
\end{equation}
with holomorphic functions $a_{k,j}$ defined in the neighborhood of $\underline{\kappa}^{(1)}$ and that vanish at $\underline{\kappa}^{(1)}$. 
The decomposition \eqref{decompositionpoly} gives (just use the definition \eqref{defWkj} and the expression of the Lagrange polynomial $P_{k,\kappa}$):
$$
W_{k,m_1} \, = \, P_{k,\kappa}({\bf T}) \, w^0 \, +\sum_{j=1}^{m_1-1} \, a_{k,j}(\kappa) \, W_{k,j} \, ,
$$
and we now apply the Cauchy-Schwarz inequality twice to get:
$$
|W_{k,m_1}|^2 \, \le \, 2 \, |P_{k,\kappa}({\bf T}) \, w^0|^2 \, +2 \, (m_1-1) \, \sum_{j=1}^{m_1-1} \, |a_{k,j}(\kappa)|^2 \, |W_{k,j}|^2 \, .
$$
Fixing from now on $\varepsilon=1/8$ and using the latter inequality in \eqref{inegfinale-edp}, we find that $q_{d,\kappa}$ is nonnegative for $\kappa$ 
sufficiently close to $\underline{\kappa}^{(1)}$ (recall that $|z_k(\kappa)|<1$ uniformly with respect to $\kappa$ in the neighborhood of $\underline{\kappa}^{(1)}$ 
since the multiple eigenvalue $\underline{z}^{(1)}$ lies in $\D$). Moreover, we observe on the defining equation \eqref{defqdmult-edp} that the Hermitian form 
$q_{d,\kappa}$ depends in a $\mathcal{C}^\infty$ way on $\kappa$ in the neighborhood of $\underline{\kappa}^{(1)}$.
\bigskip

The above analysis close to $\underline{\kappa}^{(1)}$ can be repeated word for word in the neighborhood of any other point $\underline{\kappa}^{(2)}, 
\dots, \underline{\kappa}^{(K)}$ where the dispersion relation \eqref{dispersion} has a multiple root. Now, if $\underline{\kappa} \in (\cercle)^d$ is such 
that the dispersion relation \eqref{dispersion} only has simple roots at $\kappa=\underline{\kappa}$, the analysis is much simpler since we know in that 
case that the roots $z_1,\dots,z_{s+1}$ locally depend holomorphically on $\kappa$ and the energy and dissipation forms can be simply defined as:
\begin{align*}
\forall \, (w^0,\dots,w^s) \in \C^{s+1} \, ,\quad & q_{e,\kappa}(w^0,\dots,w^s) \, := \, \sum_{k=1}^{s+1} \, |P_{k,\kappa}({\bf T}) \, w^0|^2 \, , \\
&q_{d,\kappa}(w^0,\dots,w^s) \, := \, \sum_{k=1}^{s+1} \, \big( 1 \, - \, |z_k(\kappa)|^2 \big) \, |P_{k,\kappa}({\bf T}) \, w^0|^2 \, ,
\end{align*}
with the same notation as above for the Lagrange polynomials $P_{k,\kappa}$. At this stage, we have shown that for any base point $\underline{\kappa}$ 
in the compact manifold $(\cercle)^d$, there exists an open neighborhood $\underline{\mathcal{V}}$ of $\underline{\kappa}$ in $(\cercle)^d$ and there 
exists a $\mathcal{C}^\infty$ mapping $q_{e,\kappa}$, resp. $q_{d,\kappa}$, on $\underline{\mathcal{V}}$ with values in the set of positive definite, resp. 
nonnegative, Hermitian forms, such that the decomposition \eqref{relationLG} holds for all $\kappa \in \underline{\mathcal{V}}$ and all vectors $(w^0,\dots,w^s) 
\in \C^{s+1}$. By compactness of $(\cercle)^d$, we can take a finite covering of $(\cercle)^d$ by such neighborhoods and glue the local definitions of the 
energy and dissipation forms thanks to a subordinate partition of unity. We have thus constructed a positive definite, resp. nonnegative, Hermitian form 
$q_{e,\kappa}$, resp. $q_{d,\kappa}$, on $\C^{s+1}$ which depends in a $\mathcal{C}^\infty$ way on $\kappa \in (\cercle)^d$ and such that there holds:
\begin{multline*}
2 \, \langle \, M \, v^n,L \, v^n \, \rangle_{-\infty,+\infty} \, - \, (s+1) \, \Ng \, L\, v^n \, \Nd_{-\infty,+\infty}^2 \\
= \,  ({\bf T}-I) \, \int_{\R^d} \, q_{e,\kappa} \big( \widehat{v^n}(\xi),\dots,\widehat{v^{n+s}}(\xi) \big) \, \dfrac{{\rm d}\xi}{(2\, \pi)^d} \, + \, 
 \int_{\R^d} \, q_{d,\kappa} \big( \widehat{v^n}(\xi),\dots,\widehat{v^{n+s}}(\xi) \big) \, \dfrac{{\rm d}\xi}{(2\, \pi)^d} \, ,
\end{multline*}
where we recall that $\kappa$ is a placeholder for $(\exp (i \, \xi_1 \, \Delta x_1),\dots,\exp (i \, \xi_d \, \Delta x_d))$. The conclusion of Proposition \ref{prop2} 
follows as in \cite{jfcX} by a standard compactness argument for showing continuity of the quadratic forms $E$ and $D$, and coercivity for $E$.
\end{proof}

The $\mathcal{C}^\infty$ regularity of the Hermitian forms $q_{e,\kappa},q_{d,\kappa}$ with respect to $\kappa$ is not needed in the proof of Proposition 
\ref{prop2} (continuity with respect to $\kappa$ would be enough) but we have paid attention to that particular issue since it is a crucial step for later extending 
this construction to variable coefficients problems and applying symbolic calculus rules as in \cite{laxnirenberg}. This is left to a future work.

\section{Semigroup estimates for fully discrete initial boundary value problems}
\label{section3}

It remains to prove Theorem \ref{mainthm} with the help of Proposition \ref{prop2}. The strategy is exactly the same as in \cite{jfcX} since the analysis in that 
earlier work shows that the cornerstone of the proof of Theorem \ref{mainthm} is the existence of a multiplier for the fully discrete Cauchy problem on $\Z^d$. 
Let us emphasize that the relation \eqref{egaliteLG} is of the exact same form as in \cite{jfcX}. The multiplier $M\, v^n$ is the same. The only difference is in the 
definition of the energy and dissipation forms $E$ and $D$, but their precise expression is not useful in what follows. What matters is that $D$ is nonnegative, 
and $E$ is coercive and therefore yields a control of $\ell^2$ norms on $\Z^d$. Hence we can apply the same arguments as in \cite{jfcX} as long as the proof 
of Theorem \ref{mainthm} only uses the result of Proposition \ref{prop2} and not the behavior of the roots of the dispersion 
relation \eqref{dispersion}. We thus follow the proof of \cite[Theorem 1]{jfcX} and explain where the same arguments can be applied without any modification.

\subsection{The case with zero initial data}

The first step in \cite{jfcX} is to prove the validity of \eqref{estim1d} for zero initial data ($f^0=\cdots=f^s=0$ in \eqref{numibvp}). This part of the proof only uses 
the relation \eqref{egaliteLG} and the fact that the multiplier $M$ has the same stencil as the original difference operator $L$. Hence we can repeat the arguments 
in \cite{jfcX} word for word and obtain the validity of \eqref{estim1d} when the iteration \eqref{numibvp} is considered with zero initial data. It then remains to 
consider \eqref{numibvp} with nonzero initial data in $\ell^2$ and zero interior/boundary forcing terms.

\subsection{Construction of dissipative boundary conditions}

This was the most technical step of the analysis in \cite{jfcX}. The goal here is to construct an auxiliary set of numerical boundary conditions for which, with 
arbitrary initial data in $\ell^2$, we can derive an optimal semigroup estimate and a trace estimate for the solution. Our result here is the same as in \cite{jfcX} 
but it now holds in the broader framework of Assumption \ref{assumption1}. (Theorem \ref{absorbing} is the place where Assumption \ref{assumption2} is 
needed.)

\begin{theorem}
\label{absorbing}
Let Assumptions \ref{assumption0}, \ref{assumption1} and \ref{assumption2} be satisfied. Then for all $P_1 \in \N$, there exists a constant $C_{P_1}>0$ such 
that, for all initial data $f^0,\dots,f^s \in \ell^2 (\Z^d)$ and for all source term $(g_j^n)_{j_1 \le 0,j' \in \Z^{d-1},n \ge s+1}$ that satisfies the integrability condition:
\begin{equation*}
\forall \, \Gamma \, > \, 0 \, ,\quad \sum_{n\ge s+1} \, {\rm e}^{-2\, \Gamma \, n} \, \sum_{j_1 \le 0} \, \| \, g_{j_1,\cdot}^n \, \|_{\ell^2(\Z^{d-1})}^2 \, < \, + \, \infty \, ,
\end{equation*}
there exists a unique sequence $(u_j^n)_{j \in \Z^d,n\in \N}$ in $\ell^2 (\Z^d)^\N$ solution to the iteration
\begin{equation}
\label{numabsorbing}
\begin{cases}
L \, u_j^n \, = \, 0 \, ,& j \in \Z^d \, ,\quad j_1 \ge 1\, ,\quad n \ge 0 \, ,\\
M \, u_j^n \, = \, g_j^{n+s+1} \, ,& j \in \Z^d \, ,\quad j_1 \le 0\, ,\quad n \ge 0 \, ,\\
u_j^n \, = \, f_j^n \, ,& j \in \Z^d \, ,\quad n \, = \, 0,\dots,s \, .
\end{cases}
\end{equation}
Moreover for all $\gamma>0$ and $\Delta t \in (0,1]$, this solution satisfies
\begin{multline}
\label{estimabsorb}
\sup_{n \ge 0} \, {\rm e}^{-2\, \gamma \, n\, \Delta t} \, \Ng \, u^n \, \Nd_{-\infty,+\infty}^2 \, + \, \dfrac{\gamma}{\gamma \, \Delta t+1} \, 
\sum_{n\ge 0} \, \Delta t \, {\rm e}^{-2\, \gamma \, n\, \Delta t} \, \Ng \, u^n \, \Nd_{-\infty,+\infty}^2 \\
+ \, \sum_{n\ge 0} \, \Delta t \, {\rm e}^{-2\, \gamma \, n\, \Delta t} \, \sum_{j_1=1-r_1}^{P_1} \, \| \, u_{j_1,\cdot}^n \, \|_{\ell^2(\Z^{d-1})}^2 \\
\le \, C_{P_1} \, \left\{ \, \sum_{\sigma=0}^s \, \Ng \, f^\sigma \, \Nd_{-\infty,+\infty}^2 \, + \, 
\sum_{n\ge s+1} \, \Delta t \, {\rm e}^{-2\, \gamma \, n\, \Delta t} \, \sum_{j_1 \le 0} \, \| \, g_{j_1,\cdot}^n \, \|_{\ell^2(\Z^{d-1})}^2 \right\} \, .
\end{multline}
\end{theorem}

\begin{proof}
Unsurprisingly, most of the proof of Theorem \ref{absorbing} is the same as in \cite{jfcX} but there is one specific point where the behavior of the roots 
to the dispersion relation \eqref{dispersion} is used so we review the main steps of the proof and simply refer to \cite{jfcX} when no modification is needed. 
First, the existence and uniqueness of a solution to \eqref{numabsorbing} follows from the invertibility of $Q_{s+1}$ on $\ell^2(\Z^d)$.  Then, using 
Proposition \ref{prop1}, we can derive the same estimate as in \cite{jfcX} for the solution to \eqref{numabsorbing}:
\begin{multline}
\label{estim1}
\sup_{n \ge 0} \, {\rm e}^{-2\, \gamma \, n\, \Delta t} \, \Ng \, u^n \, \Nd_{-\infty,+\infty}^2 \, + \, 
\dfrac{\gamma}{\gamma \, \Delta t +1} \, \sum_{n\ge 0} \, \Delta t \, {\rm e}^{-2\, \gamma \, n\, \Delta t} \, \Ng \, u^n \, \Nd_{-\infty,+\infty}^2 \\
+ \, \sum_{n\ge 0} \, \Delta t \, {\rm e}^{-2\, \gamma \, (n+s+1)\, \Delta t} \, \sum_{j_1 \in \Z} \, \| \, L \, u_{j_1,\cdot}^n \, \|_{\ell^2(\Z^{d-1})}^2 \\
\le \, C \, \left\{ \sum_{\sigma=0}^s \, \Ng \, f^\sigma \, \Nd_{-\infty,+\infty}^2 \, + \, 
\sum_{n\ge s+1} \, \Delta t \, {\rm e}^{-2\, \gamma \, n\, \Delta t} \, \sum_{j_1 \le 0} \, \| \, g_{j_1,\cdot}^n \, \|_{\ell^2(\Z^{d-1})}^2 \right\} \, ,
\end{multline}
where the constant $C$ is independent of $\gamma$, $\Delta t$ and on the source terms in \eqref{numabsorbing}. It remains to derive the trace estimate 
for the solution $(u_j^n)$ to \eqref{numabsorbing} (that is showing that the third term in the sum on the left hand side of the inequality \eqref{estimabsorb} 
is controlled by the right hand side).
\bigskip

The derivation of the trace estimate when $\gamma \, \Delta t$ is large enough is done as in \cite{jfcX} since it only uses the invertibility of the operator 
$Q_{s+1}$ on $\ell^2(\Z^d)$. We can thus assume from now on $\gamma \, \Delta t \in (0,\ln R_0]$ for some fixed constant $R_0>1$. Then we can 
deduce from \eqref{estim1} that for any $j_1 \in \Z$, the Laplace-Fourier transform $\widehat{u_{j_1}}$ of the step function
\begin{equation*}
u_{j_1} \quad : \quad (t,y) \in \R^+ \times \R^{d-1} \longmapsto u_j^n \quad \text{\rm if } (t,y) \in [n\, \Delta t,(n+1) \, \Delta t) 
\times \prod_{k=2}^d \, [j_k \, \Delta x_k,(j_k+1) \, \Delta x_k) \, ,
\end{equation*}
is well-defined on the half-space $\{ \tau \in \C \, , \, \text{\rm Re } \tau >0 \} \times \R^{d-1}$. The dual variables to $(t,y)$ are denoted $\tau = \gamma 
+ i \, \theta$, $\gamma>0$, and $\eta =(\eta_2,\dots,\eta_d) \in \R^{d-1}$. We also use below the notation $\eta_\Delta := (\eta_2 \, \Delta x_2, \dots, 
\eta_d \, \Delta x_d)$. The following result, which is proved in \cite{jfcX}, is used here as a blackbox since its proof is merely based on the validity of 
\eqref{estim1} and Plancherel's Theorem.

\begin{lemma}
\label{lem3'}
With $R_0>1$ fixed as above, there exists a constant $C>0$ such that for all $\gamma>0$ and $\Delta t \in (0,1]$ satisfying $\gamma \, \Delta t \in (0,\ln R_0]$, 
there holds
\begin{multline}
\label{estimlem3'}
\sum_{j_1 \in \Z} \, \int_{\R \times \R^{d-1}} \, \left| \, \sum_{\ell_1=-r_1}^{p_1} \, a_{\ell_1} \big( {\rm e}^{(\gamma +i\, \theta) \, \Delta t},\eta_\Delta \big) \, 
\widehat{u_{j_1+\ell_1}}(\gamma +i\, \theta,\eta) \, \right|^2 \, {\rm d}\theta \, {\rm d}\eta \\
+ \, \sum_{j_1 \le 0} \, \int_{\R \times \R^{d-1}} \, \left| \, \sum_{\ell_1=-r_1}^{p_1} \, {\rm e}^{(\gamma +i\, \theta) \, \Delta t} \, 
\partial_z a_{\ell_1} \big( {\rm e}^{(\gamma +i\, \theta) \, \Delta t},\eta_\Delta \big) \, \widehat{u_{j_1+\ell_1}}(\gamma +i\, \theta,\eta) \, \right|^2 \, {\rm d}\theta \, {\rm d}\eta \\
\le \, C \, \left\{ \, \sum_{\sigma=0}^s \, \Ng \, f^\sigma \, \Nd_{-\infty,+\infty}^2 \, + \, 
\sum_{n\ge s+1} \, \Delta t \, {\rm e}^{-2\, \gamma \, n\, \Delta t} \, \sum_{j_1 \le 0} \, \| \, g_{j_1,\cdot}^n \, \|_{\ell^2(\Z^{d-1})}^2 \, \right\} \, .
\end{multline}
Recall that the functions $a_{\ell_1}$, $\ell_1=-r_1,\dots,p_1$, are defined in \eqref{defA-d}.
\end{lemma}

\noindent The conclusion now relies on the following crucial result. (This is the place where the behavior of the roots to the dispersion relation \eqref{dispersion} 
matters, and where we therefore need to be careful.)

\begin{lemma}[The trace estimate]
\label{lem3}
Let Assumptions \ref{assumption0}, \ref{assumption1} and \ref{assumption2} be satisfied. Let $R_0>1$ be fixed as above and let $P_1 \in \N$. Then there exists a 
constant $C_{P_1}>0$ such that for all $z \in \U$ with $|z| \le R_0$, for all $\eta \in \R^{d-1}$ and for all sequence $(w_{j_1})_{j_1 \in \Z} \in \ell^2(\Z;\C)$, there holds
\begin{equation}
\label{estimlem3}
\sum_{j_1=-r_1-p_1}^{P_1} \, |w_{j_1}|^2 \, \le \, C_{P_1} \, \left\{ \, \sum_{j_1 \in \Z} \, \left| \, \sum_{\ell_1=-r_1}^{p_1} a_{\ell_1}(z,\eta_\Delta) \, w_{j_1+\ell_1} \, \right|^2 
\, + \, \sum_{j_1 \le 0} \, \left| \, \sum_{\ell_1=-r_1}^{p_1} \, z \, \partial_z a_{\ell_1}(z,\eta_\Delta) \, w_{j_1+\ell_1} \, \right|^2 \, \right\} \, .
\end{equation}
\end{lemma}

\noindent As in \cite{jfcX}, Lemma \ref{lem3} yields the conclusion of Theorem \ref{absorbing} by integrating \eqref{estimlem3} for the sequence 
$(\widehat{u_{j_1}}(\gamma +i\, \theta,\eta))_{j_1 \in \Z}$ with respect to $(\theta,\eta)$ (taking $z={\rm e}^{(\gamma +i\, \theta) \, \Delta t}$ accordingly), using 
the inequality \eqref{estimlem3'} from Lemma \ref{lem3'} and applying Plancherel's Theorem. We thus focus on the proof of Lemma \ref{lem3} from now on.
\end{proof}

\begin{proof}[Proof of Lemma \ref{lem3}]
We reproduce most of the proof that can already be found in \cite{jfcX} in order to highlight where Assumption \ref{assumption1} (in its new form) is used. We argue 
by contradiction and assume that the conclusion to Lemma \ref{lem3} does not hold. Therefore, up to normalizing and extracting subsequences, there exist three 
sequences (indexed by $k \in \N$):
\begin{itemize}
 \item a sequence $(w^k)_{k \in \N}$ with values in $\ell^2(\Z;\C)$ such that $(w_{-r_1-p_1}^k,\dots,w_{P_1}^k)$ belongs to the unit sphere of $\C^{P_1+r_1+p_1+1}$ 
 for all $k$, and $(w_{-r_1-p_1}^k,\dots,w_{P_1}^k)$ converges towards $(\underline{w}_{-r_1-p_1},\dots,\underline{w}_{P_1})$ as $k$ tends to infinity,
 
 \item a sequence $(z^k)_{k \in \N}$ with values in $\U \cap \{ \zeta \in \C \, , \, |\zeta| \le R_0 \}$, which converges towards $\underline{z} \in \Ubar$,
 
 \item a sequence $(\eta^k)_{k \in \N}$ with values in $[0,2\, \pi]^{d-1}$, which converges towards $\underline{\eta} \in [0,2\, \pi]^{d-1}$,
\end{itemize}
and these sequences satisfy:
\begin{equation}
\label{lem3-1}
\lim_{k \rightarrow +\infty} \quad \sum_{j_1 \in \Z} \, \left| \, \sum_{\ell_1=-r_1}^{p_1} \, a_{\ell_1}(z^k,\eta^k) \, w^k_{j_1+\ell_1} \, \right|^2 \, + \, 
\sum_{j_1 \le 0} \, \left| \, \sum_{\ell_1=-r_1}^{p_1} \, z^k \, \partial_z a_{\ell_1}(z^k,\eta^k) \, w^k_{j_1+\ell_1} \, \right|^2 \, = \, 0 \, .
\end{equation}
We are going to show that \eqref{lem3-1} implies that the vector $(\underline{w}_{-r_1-p_1},\dots,\underline{w}_{P_1})$ must be zero, which will yield a contradiction 
since this vector has norm $1$.
\bigskip

$\bullet$ We already know that $(w_{-r_1-p_1}^k,\dots,w_{P_1}^k)$ converges towards $(\underline{w}_{-r_1-p_1},\dots,\underline{w}_{P_1})$ as $k$ tends to infinity, 
and arguing by induction as in \cite{jfcX}, we can show that \eqref{lem3-1} and Assumption \ref{assumption2} imply that each component $(w^k_{j_1})_{k \in \N}$, 
$j_1 \in \Z$, has a limit as $k$ tends to infinity. This limit is denoted $\underline{w}_{j_1}$ for any $j_1 \in \Z$. Then \eqref{lem3-1} implies that the sequence 
$\underline{w}$, which does not necessarily belong to $\ell^2(\Z;\C)$, satisfies the two recurrence relations (observe that the recurrence relation \eqref{induction2} 
only holds on $(-\infty,0)$ and not on $\Z$):
\begin{align}
\forall \, j_1 \in \Z \, ,\quad & \sum_{\ell_1=-r_1}^{p_1} \, a_{\ell_1}(\underline{z},\underline{\eta}) \, \underline{w}_{j_1+\ell_1} \, = \, 0 \, ,\label{induction1} \\
\forall \, j_1 \le 0 \, ,\quad & \sum_{\ell_1=-r_1}^{p_1} \, \underline{z} \, \partial_z a_{\ell_1}(\underline{z},\underline{\eta}) \, \underline{w}_{j_1+\ell_1} \, = \, 0 
\, .\label{induction2}
\end{align}
\bigskip

$\bullet$ We define the source terms:
\begin{equation*}
\forall \, j_1 \in \Z \, ,\quad F_{j_1}^k \, := \, \sum_{\ell_1=-r_1}^{p_1} \, a_{\ell_1}(z^k,\eta^k) \, w^k_{j_1+\ell_1} \, ,\quad 
G_{j_1}^k \, := \, \sum_{\ell_1=-r_1}^{p_1} \, z^k \, \partial_z a_{\ell_1}(z^k,\eta^k) \, w^k_{j_1+\ell_1} \, ,
\end{equation*}
which, according to \eqref{lem3-1}, satisfy
\begin{equation}
\label{lem3-2}
\lim_{k \rightarrow 0} \quad \sum_{j_1 \in \Z} \, | \, F_{j_1}^k \, |^2 \, = \, 0 \, ,\quad 
\lim_{k \rightarrow 0} \quad \sum_{j_1 \le 0} \, | \, G_{j_1}^k \, |^2 \, = \, 0 \, .
\end{equation}
We also introduce the vectors (here $T$ denotes transposition)
\begin{equation*}
\forall \, j_1 \in \Z \, ,\quad W_{j_1}^k \, := \, \Big( \, w^k_{j_1+p_1},\dots,w^k_{j_1+1-r_1} \, \Big)^T \, ,\quad 
\underline{W}_{j_1} \, := \, \Big( \, \underline{w}_{j_1+p_1},\dots,\underline{w}_{j_1+1-r_1} \, \Big)^T \, ,
\end{equation*}
and the matrices:
\begin{align}
\LL (z,\eta) & \, := \, \begin{pmatrix}
-a_{p_1-1} (z,\eta)/a_{p_1} (z,\eta) & \dots & \dots & -a_{-r_1} (z,\eta)/a_{p_1} (z,\eta) \\
1 & 0 & \dots & 0 \\
0 & \ddots & \ddots & \vdots \\
0 & 0 & 1 & 0 \end{pmatrix} \in {\mathcal M}_{p_1+r_1}(\C) \, ,\label{defL} \\
\M (z,\eta) & \, := \, \begin{pmatrix}
-\partial_z a_{p_1-1} (z,\eta)/\partial_z a_{p_1} (z,\eta) & \dots & \dots & 
-\partial_z a_{-r_1} (z,\eta)/\partial_z a_{p_1} (z,\eta) \\
1 & 0 & \dots & 0 \\
0 & \ddots & \ddots & \vdots \\
0 & 0 & 1 & 0 \end{pmatrix} \in {\mathcal M}_{p_1+r_1}(\C) \, .\label{defM}
\end{align}
The matrix $\LL$ is well-defined on $\Ubar \times \R^{d-1}$ thanks to Assumption \ref{assumption2}. The matrix $\M$ is also well-defined on $\Ubar \times \R^{d-1}$ 
because for any $\eta \in \R^{d-1}$, Assumption \ref{assumption2} asserts that $a_{p_1}(\cdot,\eta)$ is a nonconstant polynomial whose roots lie in $\D$. From the 
Gauss-Lucas Theorem, the roots of $\partial_z a_{p_1}(\cdot,\eta)$ lie in the convex hull of those of $a_{p_1}(\cdot,\eta)$, hence in $\D$. Therefore $\partial_z 
a_{p_1}(\cdot,\eta)$ does not vanish on $\Ubar$. In the same way, $\partial_z a_{-r_1}(\cdot,\eta)$ does not vanish on $\Ubar$.

With our above notation, the vectors $W_{j_1}^k$, $\underline{W}_{j_1}$, satisfy the one step recurrence relations:
\begin{align}
\forall \, j_1 \in \Z \, ,\quad W_{j_1+1}^k & \, = \, \LL(z^k,\eta^k) \, W_{j_1}^k \, + \, \Big( \, F^k_{j_1+1}/a_{p_1} (z^k,\eta^k),0,\dots,0 \, \Big)^T \, ,\label{induction1'} \\
\underline{W}_{j_1+1} & \, = \, \LL(\underline{z},\underline{\eta}) \, \underline{W}_{j_1} \, , \label{induction1''} \\
\forall \, j_1 \le -1 \, ,\quad W_{j_1+1}^k & \, = \, \M(z^k,\eta^k) \, W_{j_1}^k \, + \, \Big( \, G^k_{j_1+1}/(z^k \, \partial_z a_{p_1} (z^k,\eta^k)),0,\dots,0 \, \Big)^T 
\, ,\label{induction2'} \\
\underline{W}_{j_1+1} & \, = \, \M(\underline{z},\underline{\eta}) \, \underline{W}_{j_1} \, .\label{induction2''}
\end{align}
The recurrence relations \eqref{induction1''}, \eqref{induction2''} are just an equivalent way of writing \eqref{induction1}, \eqref{induction2}.
\bigskip

$\bullet$ From Assumption \ref{assumption2} and the above application of the Gauss-Lucas Theorem, we already know that both matrices $\LL(z,\eta)$ and $\M(z,\eta)$ 
are invertible for $(z,\eta) \in \Ubar \times \R^{d-1}$. Furthermore, a quick analysis shows that $\kappa \in \C \setminus \{ 0 \}$ is an eigenvalue of $\LL(z,\eta)$ if and only 
if $z$ is a solution to the dispersion relation \eqref{dispersion}. Assumption \ref{assumption1} therefore shows that $\LL(z,\eta)$ has no eigenvalue on $\cercle$ for $(z,\eta) 
\in \U \times \R^{d-1}$ for otherwise the von Neumann condition would not hold. (This eigenvalue splitting property dates back at least to \cite{kreiss1}.) However, central 
eigenvalues on $\cercle$ may occur for $\LL$ when $z$ belongs to $\cercle$ (see \cite{jfcnotes} for a thorough analysis of the leap-frog scheme).

As in \cite{jfcX}, the crucial point for proving Lemma \ref{lem3} is that Assumption \ref{assumption1} in its new form still precludes central eigenvalues of $\M$ for all 
$z \in \Ubar$. Namely, let us show that for all $z \in \Ubar$ and all $\eta \in \R^{d-1}$, $\M(z,\eta)$ has no eigenvalue on $\cercle$. This property holds because otherwise, 
for some $(z,\eta) \in \Ubar \times \R^{d-1}$, there would exist a root $\kappa_1 \in \cercle$ to the characteristic polynomial of $\M(z,\eta)$, that is (up to multiplying by a 
nonzero factor):
\begin{equation*}
\sum_{\ell_1=-r_1}^{p_1} \, z \, \partial_z a_{\ell_1}(z,\eta) \, \kappa_1^{\ell_1} \, = \, 0 \, .
\end{equation*}
For convenience, the coordinates of $\eta$ are denoted $(\eta_2,\dots,\eta_d)$. Using the definition \eqref{defA-d} of $a_{\ell_1}$, and defining $\kappa := 
(\kappa_1,{\rm e}^{i\, \eta_2},\dots,{\rm e}^{i\, \eta_d}) \in (\cercle)^d$, we have found a root $z \in \Ubar$ to the relation
\begin{equation}
\label{dispersionder}
\sum_{\sigma=1}^{s+1} \sigma \, \widehat{Q_\sigma} (\kappa) \, z^{\sigma-1} \, = \, 0 \, .
\end{equation}
This is where the new form of Assumption \ref{assumption1} matters. Namely, we know that for all $\kappa \in (\cercle)^d$, the roots of the polynomial equation 
\eqref{dispersion} lie in $\Dbar$ and if there are roots on the boundary $\cercle$, then they must necessarily be simple. Applying again the Gauss-Lucas Theorem, 
we know that the roots to \eqref{dispersionder} lie in the convex hull of those to \eqref{dispersion} and therefore belong to $\D$ (because the only possibility for 
\eqref{dispersionder} to have a root on the boundary $\cercle$ would be that \eqref{dispersion} admits a double root on $\cercle$ but this degeneracy is precluded 
by Assumption \ref{assumption1}). The Gauss-Lucas Theorem thus shows that the roots to the relation \eqref{dispersionder} do not belong to $\Ubar$. Hence $\M 
(z,\eta)$ has no eigenvalue on $\cercle$ for any $(z,\eta) \in \Ubar \times \R^{d-1}$.
\bigskip

$\bullet$ At this stage, we know that for $(z,\eta) \in \Ubar \times \R^{d-1}$, the eigenvalues of $\M(z,\eta)$ split into two groups: those in $\U$, which we call the unstable 
ones, and those in $\D$, which we call the stable ones. For $(z,\eta) \in \Ubar \times \R^{d-1}$, we then introduce the spectral projector $\Pi_\M^s(z,\eta)$, resp. $\Pi_\M^u 
(z,\eta)$, of $\M(z,\eta)$ on the generalized eigenspace associated with eigenvalues in $\D$, resp. $\U$. These projectors are analytic with respect to $(z,\eta)$ on $\Ubar 
\times \R^{d-1}$. We can integrate from $-\infty$ to $0$ the recurrence relation \eqref{induction2'} and get
\begin{equation*}
\Pi_\M^s(z^k,\eta^k) \, W_0^k \, = \, \dfrac{1}{z^k \, \partial_z a_{p_1} (z^k,\eta^k)} \, \sum_{j_1 \le 0} \, \M(z^k,\eta^k)^{|j_1|} \, \Pi_\M^s(z^k,\eta^k) \, 
\Big( \, G^k_{j_1},0,\dots,0 \, \Big)^T \, .
\end{equation*}
The projector $\Pi_\M^s$ depends analytically on $(z,\eta) \in \Ubar \times \R^{d-1}$. Furthermore, since the spectrum of $\M$ does not meet $\cercle$ for $(z,\eta) \in 
\Ubar \times \R^{d-1}$, there exists a constant $C>0$ and a parameter $\delta \in (0,1)$ that are independent of $k \in \N$ and such that
\begin{equation*}
\forall \, j_1 \le 0 \, ,\quad \| \, \M(z^k,\eta^k)^{|j_1|} \, \Pi_\M^s(z^k,\eta^k) \, \| \, \le \, C \, \delta^{|j_1|} \, .
\end{equation*}
We thus get a uniform estimate with respect to $k$:
\begin{equation*}
| \, \Pi_\M^s(z^k,\eta^k) \, W_0^k \, |^2 \, \le \, C \, \sum_{j_1 \le 0} \, | \, G^k_{j_1} \, |^2 \, .
\end{equation*}
Passing to the limit and using \eqref{lem3-2}, we get $\Pi_\M^s (\underline{z},\underline{\eta}) \, \underline{W}_0 =0$, or in other words $\underline{W}_0 
=\Pi_\M^u (\underline{z},\underline{\eta}) \, \underline{W}_0$. Furthermore, since $(\underline{W}_{j_1})_{j_1 \le 0}$ satisfies the recurrence relation 
\eqref{induction2''} with $\underline{W}_0$ in the generalized eigenspace of $\M (\underline{z},\underline{\eta})$ associated with eigenvalues in $\U$, 
we find that $(\underline{W}_{j_1})_{j_1 \le 0}$ decays exponentially at $-\infty$ and thus belongs to $\ell^2(-\infty,0)$.
\bigskip

$\bullet$ The sequence $(\underline{W}_{j_1})_{j_1 \le 0}$ satisfies both recurrence relations \eqref{induction1''} and \eqref{induction2''}, which equivalently 
means that the complex valued sequence $(\underline{w}_{j_1})_{j_1 \le 0}$ satisfies the two recurrence relations \eqref{induction1} and \eqref{induction2} 
for $j_1 \le 0$. Hence $(\underline{w}_{j_1})_{j_1 \le 0}$ satisfies the recurrence relation associated with the greatest common divisor of the polynomials 
associated with \eqref{induction1} and \eqref{induction2}. In other words, the vector $\underline{W}_0$ belongs to the generalized eigenspace (of either 
$\LL$ or $\M$) associated with the common eigenvalues of $\M(\underline{z},\underline{\eta})$ and $\LL(\underline{z},\underline{\eta})$. Since we already 
know that $\M(\underline{z},\underline{\eta})$ has no eigenvalue on $\cercle$ and that $\underline{W}_0$ belongs to the generalized eigenspace of 
$\M(\underline{z},\underline{\eta})$ associated with eigenvalues in $\U$ (the unstable ones), we can conclude that $\underline{W}_0$ also belongs to 
the generalized eigenspace of $\LL(\underline{z},\underline{\eta})$ associated with those common eigenvalues of $\M(\underline{z},\underline{\eta})$ 
and $\LL(\underline{z},\underline{\eta})$ in $\U$.

The final argument is the following. The matrix $\LL(\underline{z},\underline{\eta})$ has $N^u$ eigenvalues in $\U$, $N^s$ in $\D$ and $N^c$ 
on $\cercle$ (all eigenvalues are counted with multiplicity). (Since $\underline{z}$ may belong to $\cercle$, $N^c$ is not necessarily zero.) With 
rather obvious notations, we let $\Pi_\LL^{u,s,c}(z,\eta)$ denote the corresponding spectral projectors of $\LL$ for $(z,\eta)$ sufficiently close to 
$(\underline{z},\underline{\eta})$. In particular, the $N^u$ eigenvalues corresponding to $\Pi_\LL^u(z,\eta)$ lie in $\U$ uniformly away from $\cercle$ 
for $(z,\eta)$ sufficiently close to $(\underline{z},\underline{\eta})$. We can then integrate \eqref{induction1'} from $+\infty$ to $0$ and derive (for $k$ 
sufficiently large):
\begin{equation*}
\Pi_\LL^u(z^k,\eta^k) \, W_0^k \, = \, -\dfrac{1}{a_{p_1} (z^k,\eta^k)} \, \sum_{j_1 \ge 0} \, \LL(z^k,\eta^k)^{-j_1-1} \, \Pi_\LL^u(z^k,\eta^k) \, 
\Big( \, F^k_{j_1},0,\dots,0 \, \Big)^T \, .
\end{equation*}
Using the uniform exponential decay of $\LL(z^k,\eta^k)^{-j_1-1} \, \Pi_\LL^u(z^k,\eta^k)$ (with respect to $j_1$) and the convergence \eqref{lem3-2}, 
we finally end up with 
\begin{equation*}
\Pi_\LL^u(\underline{z},\underline{\eta}) \, \underline{W}_0 \, = \, 0 \, .
\end{equation*}
Since $\underline{W}_0$ belongs to the generalized eigenspace of $\LL$ associated with those common eigenvalues of $\M(\underline{z},\underline{\eta})$ 
and $\LL(\underline{z},\underline{\eta})$ in $\U$, we can conclude that $\underline{W}_0$ equals zero. Applying the recurrence relation \eqref{induction1''}, 
the whole sequence $(\underline{W}_{j_1})_{j_1 \in \Z}$ is zero, which yields the expected contradiction.
\end{proof}

\subsection{End of the proof}

The end of the proof of Theorem \ref{mainthm} follows, as in \cite{jfcX}, from a superposition argument, see \cite[chapter 4]{benzoni-serre} for 
a similar argument in the context of continuous problems. The solution to \eqref{numibvp} with nonzero initial data is decomposed as the sum 
of a solution to an auxiliary problem \eqref{numabsorbing} (that auxiliary problem incorporates the initial data) and of a solution to a problem of 
the form \eqref{numibvp} with zero initial data (hence our earlier treatment of that case). The analysis in \cite{jfcX} can be applied again word 
for word so we feel free to refer the reader to that earlier work.

\appendix
\section{Numerical schemes with two time levels}

As we have seen in the proof of Proposition \ref{prop2}, the construction of energy and dissipation functionals for finite difference operators is 
dictated, through the Plancherel Theorem, by the analogous construction for recurrence relations. The inconvenience in the proof of Lemma 
\ref{lem1} is that the construction of the forms $q_e$ and $q_d$ depends on whether the roots of the polynomial $P$ are simple. There is 
however one case that can be dealt with in a unified way and for which the coefficients of the forms $q_e$ and $q_d$ depend in a very simple 
and explicit way on the coefficients of $P$. Namely, we have the following result in the case of degree two polynomials\footnote{Our attempts 
to obtain an analogue of Lemma \ref{lem1'} with `explicit' Hermitian forms for degree three polynomials have been unsuccessful so far, not 
mentioning higher degrees.} (the case of degree one polynomials is actually even simpler).

\begin{lemma}[The energy-dissipation balance law for second order recurrence relations]
\label{lem1'}
Let $P=a \, X^2 +b\, X+c \in \C [X]$ be a polynomial of degree $2$ ($a \neq 0$), that satisfies the following two properties:
\begin{itemize}
 \item The roots of $P$ are located in $\Dbar$.
 \item If $P$ has a double root, then it is located in $\D$.
\end{itemize}
Then the Hermitian form $q_e$, resp. $q_d$, defined on $\C^2$ by:
\begin{align*}
\forall \, (x_1,x_2) \in \C^2 \, , \, & q_e(x_1,x_2) := 2 \, |a|^2 \, |x_2|^2 \, + \, 2 \, \text{\rm Re } \big( \overline{a \, x_2} \, b \, x_1 \big) 
\, + \, \big( |a|^2 \, + \, |c|^2 \big) \, |x_1|^2 \, ,\\
& q_d(x_1,x_2) := \big( |a|^2 \, - \, |c|^2 \big) \, |x_2|^2 \, + \, 2 \, \text{\rm Re } \big( \overline{a \, x_2} \, b \, x_1 \big) 
 \, - \, 2 \, \text{\rm Re } \big( \overline{b \, x_2} \, c \, x_1 \big) + \big( |a|^2 \, - \, |c|^2 \big) \, |x_1|^2 \, ,
\end{align*}
is positive definite, resp. nonnegative. Furthermore, for any sequence $(v^n)_{n \in \N}$ with values in $\C$, there holds:
\begin{equation}
\label{LG2pas}
\forall \, n \in \N \, , \quad 2 \, \text{\rm Re} \, \Big( \overline{{\bf T} \, (P'({\bf T}) \, v^n)} \, P({\bf T}) \, v^n \Big) \, = \, 2 \, |P({\bf T}) \, v^n|^2 \, + \, 
q_e(v^{n+1},v^{n+2}) \, - \, q_e(v^n,v^{n+1}) \, + \, q_d(v^n,v^{n+1}) \, .
\end{equation}
\end{lemma}

\noindent The defining equations for $q_e$ and $q_d$ in Lemma \ref{lem1'} show that, if $P$ is a polynomial whose coefficients are trigonometric 
polynomials on $\R^d$, then the coefficients of $q_e$ and $q_d$ can also be chosen as trigonometric polynomials on $\R^d$ (this was not the 
case, in general, with our construction from Lemma \ref{lem1}).

\begin{proof}
The validity of \eqref{LG2pas} is a mere algebra exercise. One can for instance expand the left hand side, which reads:
$$
2 \, \text{\rm Re} \, \Big( \overline{(2 \, a \, v^{n+2} \, + \, b \, v^{n+1})} \, (a \, v^{n+2} \, + \, b \, v^{n+1} \, + \, c \, v^n) \Big) \, ,
$$
and verify that it coincides with the right hand side of \eqref{LG2pas} (a good starting point for this calculation is first to subtract $2 \, |P({\bf T}) \, v^n|^2$ 
to the latter quantity and factorize $P({\bf T}) \, v^n$ within the real part before expanding). The relation \eqref{LG2pas} can be also derived by noting 
that the above forms $q_e$ and $q_d$ in Lemma \ref{lem1'} differ from those given in \cite{jfcX} (which we have recalled in the proof of Lemma \ref{lem1}) 
in terms of the Lagrange polynomials associated with $P$ by the standard telescopic ``add and subtract'' trick. Namely, if $z_1$, $z_2$ denote the two roots 
of $P$, then $q_e$ equivalently reads:
\begin{equation}
\label{defqe'}
q_e(x_1,x_2) \, = \, |a|^2 \, \big| x_2 \, - \, z_2 \, x_1 \big|^2 \, + \, |a|^2 \, \big| x_2 \, - \, z_1 \, x_1 \big|^2 \, + \, 
|a|^2 \, (1 \, - \, |z_1|^2) \, (1 \, - \, |z_2|^2) \, |x_1|^2 \, ,
\end{equation}
where the two first terms in the sum on the right hand side correspond to the Lagrange polynomials $P_1({\bf T}) x_1$ and $P_2({\bf T}) x_1$ (see 
\eqref{defqe}), and the very last term in the sum on the right hand side has been added in order to keep $q_e$ positive definite in case $z_1$ and 
$z_2$ coincide, in which case they belong to $\D$ (this last term was absent in \cite{jfcX} since the roots were assumed to be simple). The link with 
the defining equation for $q_e$ in Lemma \ref{lem1'} is made by using the relations:
$$
a \, (z_1+z_2) \, = \, -b \, , \quad a \, z_1 \, z_2 \, = \, c \, .
$$
It is clear from the above alternative definition \eqref{defqe'} that $q_e$ is positive definite under the assumptions we have made for the polynomial $P$.

Let us now turn to the dissipation form $q_d$. In agreement with the alternative expression \eqref{defqe'} for $q_e$, the reader can check that the 
form $q_d$ given in Lemma \ref{lem1'} can be alternatively defined by the expression:
\begin{multline*}
q_d(x_1,x_2) \, = \, |a|^2 \, (1 \, - \, |z_1|^2) \, \big| x_2 \, - \, z_2 \, x_1 \big|^2 \, + \, |a|^2 \, (1 \, - \, |z_2|^2) \, \big| x_2 \, - \, z_1 \, x_1 \big|^2 \\
\, + \, |a|^2 \, (1 \, - \, |z_1|^2) \, (1 \, - \, |z_2|^2) \, \big( |x_1|^2 \, - \, |x_2|^2 \big) \, ,
\end{multline*}
where the two first terms in the sum on the right hand side read as in \eqref{defqd}, and the very last term on the right hand side has been added 
in order to keep the balance law \eqref{LG2pas} valid. Expanding the square moduli in the expression of $q_d$, we find that it can be represented 
by the Hermitian matrix:
$$
|a|^2 \, \begin{bmatrix}
1 \, - \, |z_1|^2 \, |z_2|^2 & - \, \big( (1 \, - \, |z_1|^2) \, \overline{z_2} \, + \, (1 \, - \, |z_2|^2) \, \overline{z_1} \big) \\
- \, \big( (1 \, - \, |z_1|^2) \, z_2 \, + \, (1 \, - \, |z_2|^2) \, z_1 \big) & 1 \, - \, |z_1|^2 \, |z_2|^2 \end{bmatrix} \, ,
$$
whose trace is clearly nonnegative since $z_1$ and $z_2$ belong to $\Dbar$. Furthermore, up to the positive $|a|^4$ factor, its determinant equals:
$$
\big( 1 \, - \, |z_1|^2 \, |z_2|^2 \big)^2 - \big| (1 \, - \, |z_1|^2) \, z_2 \, + \, (1 \, - \, |z_2|^2) \, z_1 \big|^2 \, .
$$
Expanding the square modulus and factorizing, the latter quantity is found to be equivalently given by:
$$
\big( 1 \, - \, |z_1|^2 \big) \, \big( 1 \, - \, |z_2|^2 \big) \, \Big( 1 \, + \, |z_1|^2 \, |z_2|^2 \, - \, 2 \, \text{\rm Re} \, (\overline{z_2} \, z_1) \Big) \, ,
$$
which is bounded from below by the nonnegative quantity:
$$
\big( 1 \, - \, |z_1|^2 \big)^2 \, \big( 1 \, - \, |z_2|^2 \big)^2 \, .
$$
Hence the determinant of $q_d$ is nonnegative, so $q_d$ is nonnegative. The proof of Lemma \ref{lem1'} is complete.
\end{proof}

\noindent Lemma \ref{lem1'} has an important consequence for the Cauchy problem \eqref{numcauchy} with $s=1$ (finite difference operators with 
two time levels). Namely, if we follow the proof of Proposition \ref{prop2} with the aim of constructing some energy and dissipation functionals for 
\eqref{numcauchy}, we introduce the multiplier $M$ as in \eqref{defLM} and obtain the relation \eqref{egaliteLG}. In the case $s=1$, the polynomial 
$P_\kappa$ reads:
$$
P_\kappa(X) \, = \, \widehat{Q_2}(\kappa) \, X^2 \, + \, \widehat{Q_1}(\kappa) \, X \, + \, \widehat{Q_0}(\kappa) \, ,
$$
If the Cauchy problem \eqref{numcauchy} is $\ell^2$-stable, then the polynomial $P_\kappa$ satisfies the conditions of Lemma \ref{lem1'} for any 
$\kappa \in (\cercle)^d$. Hence we can apply Lemma \ref{lem1'} and rewrite \eqref{egaliteLG} as:
\begin{equation}
\label{enerdissip2-1}
2 \, \langle \, M \, v^n,L \, v^n \, \rangle_{-\infty,+\infty} \, = \, 2 \, \Ng \, L\, v^n \, \Nd_{-\infty,+\infty}^2 \, + \, 
E(v^{n+1},v^{n+2}) \, - \, E(v^n,v^{n+1}) \, + \, D(v^n,v^{n+1}) \, ,
\end{equation}
with (here we apply the Plancherel Theorem `backwards'):
\begin{align*}
E(v^n,v^{n+1}) \, := \, & \, \int_{\R^d} 2 \, |\widehat{Q_2}(\kappa)|^2 \, |\widehat{v^{n+1}}(\xi)|^2 \, + \, 
2 \, \text{\rm Re } \left( \overline{\widehat{Q_2}(\kappa) \, \widehat{v^{n+1}}(\xi)} \, \widehat{Q_1}(\kappa) \, \widehat{v^n}(\xi) \right) \\
& \quad \quad \, + \, \Big( |\widehat{Q_2}(\kappa)|^2 \, + \, |\widehat{Q_0}(\kappa)|^2 \Big) \, |\widehat{v^n}(\xi)|^2 \, \dfrac{{\rm d}\xi}{(2\, \pi)^d} \, ,\\
= \, & \, 2 \, \Ng \, Q_2 \, v^{n+1} \, \Nd_{-\infty,+\infty}^2 \, + \, 2 \, \langle \, Q_2 \, v^{n+1},Q_1 \, v^n \, \rangle_{-\infty,+\infty} 
\, + \,  \Ng \, Q_2 \, v^n \, \Nd_{-\infty,+\infty}^2 \, + \,  \Ng \, Q_0 \, v^n \, \Nd_{-\infty,+\infty}^2 \, ,
\end{align*}
and, similarly:
\begin{multline*}
D(v^n,v^{n+1}) \, := \, \Ng \, Q_2 \, v^{n+1} \, \Nd_{-\infty,+\infty}^2 \, - \,  \Ng \, Q_0 \, v^{n+1} \, \Nd_{-\infty,+\infty}^2 \\
+ \, 2 \, \langle \, Q_2 \, v^{n+1},Q_1 \, v^n \, \rangle_{-\infty,+\infty} \, - \, 2 \, \langle \, Q_1 \, v^{n+1},Q_0 \, v^n \, \rangle_{-\infty,+\infty} 
\, + \, \Ng \, Q_2 \, v^n \, \Nd_{-\infty,+\infty}^2 \, - \,  \Ng \, Q_0 \, v^n \, \Nd_{-\infty,+\infty}^2 \, .
\end{multline*}
The interesting feature of these expressions is that both $E$ and $D$ correspond to the sum, with respect to $j \in \Z^d$, of local energy and dissipation 
densities $E_j(v^n,v^{n+1})$, resp. $D_j(v^n,v^{n+1})$, which depend on finitely many values of the sequences $v^n,v^{n+1}$ near $j$. For instance, the 
local density $E_j(v^n,v^{n+1})$ can be defined by:
$$
E_j(v^n,v^{n+1}) \, := \, 2 \, |Q_2 \, v_j^{n+1}|^2 \, + \, 2 \, (Q_2 \, v_j^{n+1}) \, (Q_1 \, v_j^n) \, + \, |Q_2 \, v_j^n|^2 \, + \,  |Q_0 \, v_j^n|^2 \, .
$$
Hence there is now a genuine hope to extend the definition of $E$ and $D$ to more general domains (by means of sums of local quantities which do not 
rely on the Fourier transform) and/or to take the energy-dissipation balance law \eqref{enerdissip2-1} as a starting point for deriving stability estimates for 
finite volume space discretizations on unstructured meshes. This is left to a future work.

\bibliographystyle{alpha}
\bibliography{LG}
\end{document}